\documentclass[12pt]{article}    

\usepackage[top=1in, bottom=1in, left=1in, right=1in]{geometry}
\usepackage{graphicx}
\usepackage{amsmath,amsfonts,amsthm,mathrsfs,amssymb,cite}
\usepackage{epstopdf}

\numberwithin{equation}{section}

\newtheorem{thm}{Theorem}[section]
\newtheorem{cor}[thm]{Corollary} 
\newtheorem{lem}{Lemma}[section]
\newtheorem{prop}{Proposition}[section]

\theoremstyle{definition}
\newtheorem{defn}{Definition}[section]

\theoremstyle{remark}
\newtheorem{rem}{Remark}[section]
\newtheorem{cond}{Condition}[section]

\usepackage{xcolor}

\newcommand{\condt}{a}
\newcommand{\condcor}{\gamma}
\newcommand{\pare}[1]{\left(#1\right)}
\newcommand{\Pare}[1]{\big(#1\big)}
\newcommand{\abs}[1]{\left\lvert #1 \right\rvert}

\newcommand{\RR}{\mathbb{R}}
\newcommand{\Ss}{\mathbb{S}}
\newcommand{\CC}{\mathbb{C}}
\newcommand{\im}{\mathtt{i}}

\newcommand{\In}{\mathrm{in}}
\newcommand{\Sc}{\mathrm{sc}}

\DeclareMathOperator{\supp}{\mathrm{supp}}
\DeclareMathOperator{\Id}{\mathrm{Id.}}

\begin{document} 
\title{On Corner Scattering for Operators of Divergence Form and Applications to Inverse Scattering}
\author{ Fioralba Cakoni \ and Jingni Xiao\footnote{Department of Mathematics, Rutgers University, Piscataway, NJ 08854-8019, USA,
 emails:  fc292@math.rutgers.edu,  jingni.xiao@rutgers.edu }}
\date{}

\maketitle

	
\begin{abstract}
	We consider the scattering problem governed by the Helmholtz equation with inhomogeneity  in both ``conductivity'' in the divergence form and  ``potential'' in the lower order term.
	The support of the inhomogeneity is assumed to contain a convex corner. We prove that, due to the presence of such corner under appropriate assumptions on the potential and  conductivity in the vicinity of the corner,  any incident field scatters. Based on corner scattering analysis we present  a uniqueness  result on determination of  the polygonal convex hull of the support of admissible inhomogeneities, from scattering data corresponding to one single incident wave. These results require  only certain regularity around the corner for the coefficients modeling the  inhomogeneity, whereas away from the corner they can be quite general. Our main results on scattering and inverse scattering are established for $\RR^2$, while some analytic tools are developed in any dimension $n\geq 2$.
\end{abstract}

\noindent{\bf Key words:}  Inverse medium scattering, corner scattering, transmission eigenvalues, non-scattering wave numbers, shape determination

\noindent{\bf AMS subject classifications:} 35R30, 35J25, 35P25, 35P05, 81U40

\section{Introduction}
The existence of non-scattering wave numbers (otherwise referred to as frequencies or energies) in the scattering by inhomogeneous media, remains a perplexing question despite the recent progress starting with the pioneering paper \cite{BPS14}. A non-scattering wave number for a given inhomogeneity corresponds to the frequency for which there exists an incident wave that is not scattered by the media. It is easily seen that non-scattering wave numbers, if exist, are examples of the so-called transmission eigenvalues \cite{CakoniColtonHaddar2016}. The latter are the eigenvalues of a non-selfajoint eigenvalue problem with a deceptively simple formulation, given by two different elliptic equations in a bounded domain that coincides with the support of  inhomogeneity and sharing the same Cauchy data on the boundary. It has been shown that, under suitable assumptions on the contrast of scattering media, real transmission eigenvalues exist \cite{CakGinHad} and they can be seen in the scattering data \cite{CakColHad,KiLe}. However, for a transmission eigenvalue to be non-scattering wave number, one must be able to extend the part of the transmission eigenfunction corresponding to the background  equation as a solution to the background equation in the entire space, which is not a trivial question in general.  It is well-known that real transmission eigenvalues corresponding to a spherically stratified inhomogeneity are non-scattering wave numbers \cite{CoM88} and furthermore, all transmission eigenvalues uniquely determine the spherically stratified refractive index \cite{colyj}. However beyond the case of spherically stratified media, there is no other known type of bounded supported inhomogeneities for which non-scattering wave numbers are proven to exist. We remark that the existence of non-scattering waves has been observed in scattering problems for waveguides \cite{BoChe}. The connection between transmission eigenvalues and non-scattering energies is also studied in  some cases of hyperbolic geometries \cite{BlV18,CaCh}.

The absence of non-scattering wave numbers was first shown in \cite{BPS14} for inhomogeneities whose support contains a right corner.  {It was further studied in  \cite{PSV17} for convex conic corners, in \cite{ElH15} for 2D corners and 3D edges, in \cite{ElH18} for more general corners and edges, in \cite{LHY18} for weakly singular interfaces in 2D, in \cite{CDH18} for conductive boundary problem and in \cite{Bla18} for the source problem.} Further, in \cite{BlL18} authors proved a stability type of result which connects the local curvature with the scattering amplitude. Recently, in  \cite{BLX19arXiv} and \cite{LiX17Corner} the corner scattering investigation is extended to electromagnetic inhomogeneous scattering problems. The fact that corners and edges always scatter is employed to prove that the far field pattern corresponding to one single incident wave uniquely determines the support of  a convex polygonal inhomogeneous media, see e.g. \cite{BlL16arXiv}, \cite{ElH18} and \cite{HSV16}. Related studies  \cite{BLLW17} and \cite{BlL17}  discuss the properties of the transmission eigenfunctions and their possible extension in a neighborhood of a corner.   
We would also like to mention that there are several works on propagation of singularities for solutions of the wave equation in manifolds with conic and other types of singularities, using microlocal analysis related techniques (see, e.g, \cite{Leb97,MeW04,Vas08} and the references therein). However our choice of  the approach here is determined by particularity of the question under investigation. More specifically, we are concerned with the existence of non-scattering frequencies, which is related to the behavior of eigenfunctions of the nonstandard transmission eigenvalue value problem. Hence, we do not simply analyze the scattering phenomena near a corner,  but rather our problem becomes whether certain solutions of elliptic partial differential equations can be extended outside a corner \cite{Lew59}.
Moreover, our analysis applies to $L^\infty$ coefficients and our results have applications in inverse scattering problems.

In this paper we undertake a study of corner scattering for the scalar scattering problem corresponding to inhomogeneities with contrast in both the main operator and the lower term.  {For notational simplicity, with an abuse of terminology though, we call ``conductivity" the coefficient in the main operator and ``potential" the coefficient in the lower term, throughout the paper. We prove that,
any incident wave produces non-zero scattered field in the exterior of the inhomogeneity, providing the existence of a corner at the support of the potential with non-zero contrast  where the conductivity contrast vanishes to the second order at the vertex. In addition we show that if the aperture of the corner is an irrational factor of $\pi$, we have the same nontrivial scattering result for all incident waves. Otherwise, if the conductivity has  nontrivial contrast at the corner,  or the conductivity  contrast goes to zero slower than second order at the vertex, we need to exclude a certain class of incident fields from our results. For more detailed statements we refer the reader to Theorems \ref{thm:cornerscatt} and \ref{thm:cornerscattPot} in the paper. As an application of corner scattering we discuss an approximation property  of transmission eigenfunctions by Herglotz wave functions in the presence of corners on the support of the inhomogeneity, providing more insight to this  issue already discussed in \cite{BLLW17} and \cite{BlL17}.}  

{Another main result of our paper concerns the inverse scattering problem of shape determination for inhomogeneities, for which the uniqueness is proven by applying  corner scattering analysis.  We show that scattering data corresponding to a single incident field uniquely determines the polygonal convex hull of the support of the inhomogeneity under appropriate assumptions on conductivity and potential contrasts at the corners of the polygonal. In particular, Theorem~\ref{th:uni} states the uniqueness result for inhomogeneities whose polygonal convex hull has potential jump at all corners and at the vertices the conductivity contrast vanishes to the second order. However, we remark that this uniqueness result is  in fact valid for other types of inhomogeneities. For example, we could allow that all corners of the polygonal convex hull where the conductivity has a  jump, have aperture as irrational factor of $\pi$. More generally, if two inhomogeneities lead to the same scattering data when probed by the same incident wave, we can conclude that the difference between the two convex hulls cannot contain certain types of corners.}

{Our results generalize the previous work on corner scattering and shape determination in  \cite{BPS14,ElH15,ElH18,PSV17,BlL16arXiv,HSV16}, where the authors consider only the case when the conductivity is identically one in the whole space. In particular, this is a special case  of our setting where the contrast of the conductivity vanishes to second order at the corner. Nevertheless, we  recall that  here we do not assume any additional properties of the conductivity away from the corner, besides basic ellipticity and boundedness requirements for the forward problem, making our setting much more general. For example, we allow inhomogeneities with disconnected support or with voids inside  (see e.g. Figures~\ref{conf}, \ref{adm} and \ref{uni}, for a visualization of the support of admissible inhomogeneities), or even anisotropic materials could be allowed  away from the corners.  
The setting where the conductivity possesses contrast at the corner is a novelty of this study and it presents interesting questions related to potential exclusive incident waves for special corners, which  calls for deeper understanding.

{Finally, we remark that our approach is based on asymptotic analysis on the integrals appearing in an identity which is obtained as consequence of the non-scattering phenomenon. In order to do so, it is of fundamental importance to construct the so-called Complex Geometric Optics (CGO) solutions with desired estimates for the corresponding differential operator. We develop this analytical framework for any dimension $n\geq 2$, including the construction of CGO solutions as well as the derivation of asymptotic estimates on the integrals. However,  in the analysis of corner scattering we restrict ourselves to ${\mathbb R}^2$, avoiding technicalities that higher dimensions present in a key point of our analysis, namely deducing the strictly non-zero asymptotic behavior of a certain integral.}

The paper is organized as follows. Having formulated the problem in the next section, Sections \ref{cgo} and \ref{sec:corner} are devoted to the construction of CGO solutions for the considered problem and their use to analyze the behavior of solutions of the transmission eigenvalue problem in the vicinity of a generalized corner  both in ${\mathbb R}^2$ and ${\mathbb R}^3$. In Section \ref{non-scat} we restrict ourselves to the two-dimensional case, and provide a comprehensive analysis of  ``conductivity" and ``potential'' corner scattering in Theorems~\ref{thm:cornerscatt} and Theorem \ref{thm:cornerscattPot}}, respectively. Section \ref{applic} is devoted to the aforementioned applications of corner scattering to inverse scattering theory.

\section{Formulation of the scattering problem}\label{form}
We consider the scattering problem governed by
\begin{equation}\label{eq:MainGov1}
\nabla\cdot \condt \nabla u + k^2 c u=0\quad\mbox{in $\RR^n$},
\end{equation}
where the total field $u:=u^{\In}+u^{\Sc}\in H^1_{loc}(\RR^n)$ is composed of the incident field $u^{\In}$ and the scattered field $u^{\Sc}$.
The incident field satisfies the Helmholtz equation
\begin{equation}
\Delta u^{\In}+k^2u^{\In}=0\quad\mbox{in $\RR^n$},
\end{equation} 
and the scattered field satisfies the Sommerfeld radiation condition
\begin{equation}\label{eq:Radiat}
\hat{x}\cdot\nabla u^{\Sc}-\im k u^{\Sc}=o ( |x|^{-\frac{n-1}{2}} )
\end{equation} 
uniformly with respect to $\hat x=x/|x|$. The coefficients $\condt$ and $c$ in \eqref{eq:MainGov1} representing the constitutive material properties of the media, are real valued scalar $L^{\infty}$ functions satisfying
\begin{equation}\label{eq:gammEllip}
 \condt(x)\geq a_0>0
 \qquad \mbox{for almost all $x\in\RR^n$},
\end{equation}
with a constant $a_0$ and 
\begin{equation}\label{eq:gammBound}
\supp (c-1) \cup  \supp (\condt-1)  \subseteq \overline{\Omega},
\end{equation}
where $\Omega$ is a simply connected bounded region in $\RR^n$, i.e. the inhomogeneity is included in $\Omega$ and in the background media the constitutive material properties are $a=1$ and $c=1$. We sometimes denote such an inhomogeneity as $(a, c, \Omega)$, despite the fact that the specific domain $\Omega$ could be chosen differently. Note that (\ref{eq:MainGov1}) implicitly contains the continuity of the field and co-normal derivative wherever $a$ jumps.

\noindent
The far field pattern $u^\infty(\hat x)$ of the scattered field $u^s$ is defined via the following asymptotic expansion of the scattered field 
\[
u^{\Sc}(x)=\frac{\exp(\im k|x|)}{|x|^{\frac{n-1}{2}}}u^\infty(\hat x)+O\left(\frac{1}{|x|^{\frac{n+1}{2}}}\right),\;r\to\infty
\]
where $\hat x=x/|x|$ (c.f.~\cite{CakoniColtonHaddar2016}). 
We are particularly interested in the situation when the support of the contrast $\condt-1$ and/or $c-1$ has a corner at its shape. We would like to show that when there is such a corner, then for any incident field $u^{\In}$, the scattered wave $u^{\Sc}$ cannot vanish identically outside any region containing $\Omega$, or equivalently, the far-field aplitude $u^{\infty}$ cannot be trivially zero. 
Notice that this problem, in general, cannot be transformed into the parallel study on corner scattering  for the problem governed by
\begin{equation}
\Delta v + k^2\left(\frac{c}{\condt}-\frac{1}{k^2}\frac{\Delta\condt^{1/2}}{\condt^{1/2}}\right) v=0\quad\mbox{in $\RR^n$},
\end{equation}
by the dependent change of variables $v:=\condt^{1/2}u$, since it will introduce new singularities in the above equation. Later on, we impose  some regularity assumptions on the coefficients $a$ and $c$ in a neighborhood of the corner. Nevertheless we allow for jumps on $a$ and $c$ across the boundary of $\supp(a-1)\cup\supp(c-1)$. In fact, away from the corner, $a$ can be even $L^\infty$ positive-definite matrix-valued function.

\subsection{A consequence of the non-scattering phenomenon}
We assume that for a given inhomogeneity with constitutive parameters  $\condt(x)$ and $c(x)$ there exists a nontrivial incident field $u^{\In}$ which would not be perturbed by $\condt$ and $c$ when observed by a far-field observer. 
In this case, the corresponding far-field pattern $u^{\infty}$ will vanish identically or equivalently, the scattered field $u^{\Sc}$ will be zero in the exterior of any simple-connected Lipschitz domain $\Omega$ enclosing $\supp (c-1) \cup  \supp (\condt-1)$ (see Figure \ref{conf}).
As a consequence, the following transmission eigenvalue problem is satisfied 
\begin{eqnarray}
\nabla\cdot \condt \nabla u + k^2 c u=0,&  \Delta v + k^2 v=0,&\quad\mbox{in $\Omega$},\label{eq:ITPpde}\\
u=v,&  \condt\partial_{\nu} u=\partial_{\nu}v,&\quad\mbox{on $\partial\Omega$},\label{eq:ITPbc}
\end{eqnarray}
with $v:=u^{\In}$, where $\nu$ is the outward unit normal to $\partial \Omega$.

\begin{figure}[h]
	\centerline{
		\includegraphics[width=0.22\textwidth]{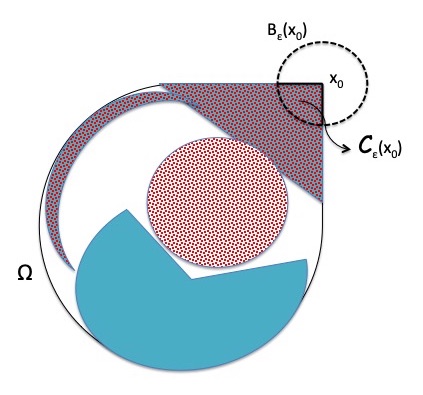}}
	\caption{Sketch of an inhomogeneity satisfying our assumptions. Dotted filling depicts $\supp(c-1)$, uniform coloring  depicts $\supp(a-1)$, while darker dotted filling depicts the  support of $\supp(c-1)\cap \supp(a- 1)$. The domain $\Omega$ must contain a corner for the contrast $c-1$ or/and $a-1$.}
	\label{conf}
\end{figure}

\begin{lem}\label{lem:Int}
	If $u$ and $v$ satisfy \eqref{eq:ITPpde}, then one has
	\begin{equation*}
	\int_{\Omega}\pare{\condt-1} \nabla v\cdot\nabla w\, dx -k^2\int_{\Omega}\pare{c-1} v w\, dx
	=\int_{\partial\Omega}\condt\partial_{\nu}w\pare{v-u}-w\pare{\partial_{\nu}v-\condt\partial_{\nu}u}\, ds
	\end{equation*}
	for any solution $w$ to
	\begin{equation*}
	\nabla\cdot \condt \nabla w + k^2c w=0,\quad \mbox{in $\Omega$}.
	\end{equation*}
\end{lem}
\begin{proof}
	Since $u$ and $w$ are both solutions to the same equation, we have from the Green's formula that
	\begin{equation*}
	\int_{\partial\Omega}\condt u \partial_{\nu}w\, ds
	=\int_{\Omega}\condt\nabla u\cdot \nabla w-k^2c uw \, dx
 	=\int_{\partial\Omega}\condt w \partial_{\nu}u\, ds.
	\end{equation*}
	Similarly we have
	\begin{equation*}
	\begin{split}
	\int_{\Omega}\condt\nabla v\cdot \nabla w \,dx&- k^2\int_{\Omega}c vw\,dx=\int_{\partial\Omega}\condt v \partial_{\nu}w\,ds
	\\&=\int_{\Omega}\nabla v\cdot \nabla w\,dx- k^2\int_{\Omega} vw\,dx
	-\int_{\partial\Omega}w \partial_{\nu}v\,ds+\int_{\partial\Omega}\condt v \partial_{\nu}w\, ds,
	\end{split}
	\end{equation*}
	where in the second identity we have used that $v$ satisfies the Helmholtz equation. 
	It is hence obtained that
	\begin{equation*}
	\begin{split}
	\int_{\Omega}\pare{\condt-1} \nabla v\cdot\nabla w\, dx& -k^2\int_{\Omega}\pare{c-1} v w\, dx
	=\int_{\partial\Omega}\condt v \partial_{\nu}w-w \partial_{\nu}v
	\\&=\int_{\partial\Omega}\condt\partial_{\nu}w\pare{v-u}-w\pare{\partial_{\nu}v-\condt\partial_{\nu}u}\, ds.
	\end{split}
	\end{equation*}
\end{proof}
The identity in the above lemma is a  fundamental tool in our forthcoming analysis of the scattering by corners.
{Later in Section \ref{sec:corner}, we analyze in details these integrals near the corner, i.e. in $B_{\varepsilon}(x_0)\cap \Omega$ as shown in Figure~\ref{conf}.}
\section{Complex rapidly decaying solutions}\label{cgo}
In this section, we shall seek for solutions to the equation
\begin{equation}\label{eq:PDE}
\nabla\cdot \gamma \nabla w + k^2 \rho w=0\quad\mbox{in $\RR^n$},
\end{equation}
which are of the form
\begin{equation}\label{eq:CGO}
w=w_{\tau}=\gamma^{-1/2}(1+r(x))e^{\eta\cdot x}.
\end{equation}
\noindent
Here, $\eta\in\CC^n$ is defined as
\begin{equation}\label{eq:eta}
\eta=-\tau \pare{d+\im d^{\perp}},
\end{equation} 
with $d,d^{\perp}\in \Ss^{n-1}$ satisfying $d\cdot d^{\perp}=0$. These solutions are referred to in the literature as Complex Geometrical Optics (CGO) solutions. 
Given $s\in\RR$ and $p\geq 1$, we recall the Bessel potential space $$H^{s,p}:=\{f\in L^p(\RR^n);\ \mathcal{F}^{-1}[(1+|\xi|^2)^{s/2}\mathcal{F}f]\in L^{p}({\mathbb R}^n)\},$$ 
where $\mathcal{F}$ and $\mathcal{F}^{-1}$ denote the Fourier transform and its inverse, respectively. 

\begin{cond}\label{cond1}
Given $s\in \RR$, the coefficients $\gamma$ and $\rho$ are such that $q:=\gamma^{-1/2}\Delta \gamma^{1/2}-k^2\rho\gamma^{-1}\in H^{s,\tilde{p}}$ and
\begin{equation}\label{eq:condqf}
\|qf\|_{H^{s,\tilde{p}}}\le C_1\|f\|_{H^{s,p}}\quad \mbox{for any $f\in H^{s,p}$},
\end{equation}
with some  $1<\tilde{p}<2$ satisfying 
\begin{equation}\label{eq1}
2/(n+1)+1/p\le 1/\tilde{p}< 2/n+\min\{1/p, s/n\}.
\end{equation}
	\end{cond}
\noindent
Some instances of feasible choices for the  parameters $s,p,\tilde p$  such that  (\ref{eq1}) holds are given next. For $s=0$ in ${\mathbb R}^2$ for any $p>3$ one can find  $1<\tilde{p}<2$, whereas in ${\mathbb R}^3$ a number $\tilde p$ in the interval $(3/2, 2)$ can be chosen for any $p>6$. If  $s=1$ is chosen, then for any $p>6/5$  in ${\mathbb R}^2$ and any  $p>2$ in ${\mathbb R}^3$ one can find  $1<\tilde{p}<2$. We will be more specific when we apply the result from this section to our inverse scattering problem.
	
\noindent
In order to see the relevance of introducing the function $q$ in Condition \ref{cond1}, we note that if we let $v:=\gamma^{1/2} w$ then $v$ satisfies
$$\Delta v -qv=0\quad\mbox{in $\RR^n$}.$$
Thus we can make use of existing results on the CGO solutions for the above equation.
\begin{prop}\label{prop:CGO}
	Given $n=2,3$, $s\in\RR$, let $\gamma$ satisfy Condition \ref{cond1} with the constant $p$ subjected to $p>1+2/(n-1)$ and $n/p<2/(n+1)+s$. Then for any $\tau>0$ large enough, 
	there is a solution $w$ to \eqref{eq:PDE} which is of the form \eqref{eq:CGO} with the residual $r$ satisfying 
	\begin{equation}\label{eq:CGOesti_tau}
	\|r\|_{H^{s,p}}=O(\tau ^{s-n/p-\sigma}),
	\end{equation} 
	with a constant $\sigma>0$ independent of $\tau$ and $r$.
\end{prop}
In order to prove Proposition~\ref{prop:CGO}, we apply the following lemma from \cite[Proposition 3.3]{PSV17}, which is based on the uniform Sobolev inequalities given in \cite{KRS87}.
\begin{lem}\label{lem:CGO_PSV}
	Suppose that $n\ge 2$, $s\in\RR$, $1<\tilde{p}<2$, and $p>1$ satisfying $2/(n+1)\le 1/\tilde{p}-1/p< 2/n$.
	Then for any $\eta\in\CC^n$ of the form \eqref{eq:eta} with $\tau$ sufficiently large,
	there is an operator 
	$
	\mathscr{G}_{\eta}: H^{s,\tilde{p}} \to H^{s,p}
	$
	which maps $f\in H^{s,\tilde{p}}$ to a solution $r=\mathscr{G}_{\eta}f$ of 
	\begin{equation}\label{eq:CGO_PSV}
	(\Delta+2\eta\cdot\nabla)r=f\quad\mbox{in $\RR^n$},
	\end{equation}
	which satisfies 
	\begin{equation*}
	\|\mathscr{G}_{\eta} f\|_{H^{s,p}}\, {\lesssim}\, \frac{\|f\|_{H^{s,\tilde{p}}}}{\tau^{2-n(1/\tilde{p}-1/p)}}.
	\end{equation*}
\end{lem} 
\noindent
We remark that in the formulation of the  lemma and in what follows  throughout the paper,  the notation ${\lesssim}$ means less than or equal to up to a constant independent of $\tau$, for $\tau$ sufficiently large.
\begin{proof}[Proof of Proposition~\ref{prop:CGO}]
	Substituting the form \eqref{eq:CGO} into the equation~\eqref{eq:PDE} yields
	\begin{equation}\label{eq:PDEresidual}
	\Delta r +2\eta\cdot\nabla r=qr + q,
	\end{equation}
	where $q$ is the function defined in Condition~\ref{cond1}.
	Conversely, one can observe that if $w$ is defined as in \eqref{eq:CGO} with $r$ satisfying \eqref{eq:PDEresidual}, then $w$ is a solution to \eqref{eq:PDE}.
	
	We construct 
	\begin{equation}\label{eq:CGOresidual}
	r:=\pare{\Id - \mathscr{G}_{\eta}q}^{-1}\mathscr{G}_{\eta} q,
	\end{equation}
	by claiming that $\pare{\Id - \mathscr{G}_{\eta}q}^{-1}\mathscr{G}_{\eta}$ is a bounded operator mapping from $H^{s,\tilde{p}}$ to $H^{s,p}$, with 
	\begin{equation*}
	\|\pare{\Id - \mathscr{G}_{\eta}q}^{-1}\mathscr{G}_{\eta}f\|_{H^{s,p}}
	\lesssim \frac{\|f\|_{H^{s,\tilde{p}}}}{\tau^{2-n(1/\tilde{p}-1/p)}}.
	\end{equation*}
	If this claim is true, one can verify straightforwardly that the function $r$ defined in \eqref{eq:CGOresidual} satisfies \eqref{eq:PDEresidual}, by using that $\mathscr{G}_{\eta}$ is a solution operator of \eqref{eq:CGO_PSV}.
	Moreover, we have 
	\begin{equation*}
	\|r\|_{H^{s,p}}
	\lesssim  \frac{\|q\|_{H^{s,\tilde{p}}}}{\tau^{2-n(1/\tilde{p}-1/p)}}
	=\frac{\|q\|_{H^{s,\tilde{p}}}}{\tau^{\sigma -s +n/p }},
	\end{equation*}
	with the constant $\sigma:=2-n/\tilde{p}+s>0$.
	We are left to prove the claim.
	
	Notice that $\mathscr{G}_{\eta}q$ is a bounded operator on $H^{s,p}$. In particular, one has
	\begin{equation*}
	\|\mathscr{G}_{\eta}q f\|_{H^{s,p}}
	\lesssim  \frac{\|q f\|_{H^{s,\tilde{p}}}}{\tau^{2-n(1/\tilde{p}-1/p)}}
	\lesssim  \frac{\|f\|_{H^{s,p}}}{\tau^{2-n(1/\tilde{p}-1/p)}},
	\end{equation*}
	for any $f\in H^{s,p}$.
	Recall that $2-n(1/\tilde{p}-1/p)$ is positive. As a consequence, the operator $\Id - \mathscr{G}_{\eta}q$ is invertible on $H^{s,p}$ with a bounded inverse for $\tau$ sufficiently large.
	Therefore we have for any $f\in H^{s,p}$ that
	\begin{equation*}
	\|\pare{\Id - \mathscr{G}_{\eta}q}^{-1}\mathscr{G}_{\eta}f\|_{H^{s,p}}
	\lesssim  \|\mathscr{G}_{\eta}f\|_{H^{s,p}}
	\lesssim  \frac{\|f\|_{H^{s,p}}}{\tau^{2-n(1/\tilde{p}-1/p)}}=\frac{\|f\|_{H^{s,p}}}{\tau^{\sigma-s+n/p}}.
	\end{equation*}
	
	The proof is complete.
\end{proof}

\section{Local analysis of solutions near a corner}\label{sec:corner}
We now use the CGO solutions introduced in Section \ref{cgo} to analyze the behavior of solutions to the partial differential equations of interest in a neighborhood of a corner, providing the bridge to final goal of understanding the scattering by an inhomogeneity whose support contains a corner. In what follows, ${\Ss^{n-1}}$ denotes the unit sphere in ${\mathbb R}^n$ $n\geq 2$, and ${\Ss_+^{n-1}}$ denotes the upper half unit sphere. 

Let $\mathcal{K}$ be a given open subset of ${\Ss_+^{n-1}}$ which is Lipschitz and simply connected. We define the (infinite and open) ``generalized cone'' $\mathcal{C}=\mathcal{C}_{\mathcal{K}}$ as $\mathcal{C}:=\{x\in\RR^n;\ \hat{x}\in\mathcal{K}\}$.
Denote $\mathcal{C}_{\epsilon}:=\mathcal{C}\cap B_{\epsilon}$ and $\mathcal{K}_{\epsilon}:=\mathcal{C}\cap\partial B_{\epsilon}$, where $B_{\epsilon}$ is the ball centered at the origin of radius $\epsilon>0$. 
Given a positive constant $\delta$, we define $\mathcal{K}'_{\delta}$ as the open set of $\Ss^{n-1}$ which is composed of all directions $d\in\Ss^{n-1}$ satisfying that 
\begin{equation}\label{eq:dProperty}
d\cdot \hat{x}>\delta>0,\quad\mbox{for all $\hat{x} \in \mathcal{K}$}.
\end{equation}
Let $\rho\in L^{\infty}(\RR^n)$ and $\condcor\in L^{\infty}(\RR^n)$ satisfy
\begin{equation}\label{eq:gammEllipCor}
0<\lambda_1\le \condcor(x) \le \lambda_2<\infty,
\quad \mbox{for almost all $x\in\mathcal{C}_{\epsilon}$}.
\end{equation}

The following result is a direct consequence of Lemma~\ref{lem:Int}.
\begin{lem}\label{lem:IntCor}
	If $u,v\in H^1(\mathcal{C}_{\epsilon})$ satisfy
	\begin{equation}\label{eq:TEVPlocal}
	\begin{split}
	\nabla\cdot \condcor \nabla u + k^2 \rho u=0,\  \Delta v + k^2 v=0,\qquad&\mbox{in $\mathcal{C}_{\epsilon}$},\\
	u=v, \quad \condcor\partial_{\nu} u=\partial_{\nu}v,\qquad&\mbox{on $\partial\mathcal{C}_{\epsilon}\setminus\mathcal{K}_{\epsilon}$},
	\end{split}
	\end{equation}
	then one has 
	\begin{equation}\label{eq:IntCor}
	\int_{\mathcal{C}_{\epsilon}}\pare{\condcor-1} \nabla v\cdot\nabla w-k^2(\rho-1) v w \,dx
	=\int_{\mathcal{K}_{\epsilon}}\condcor\partial_{\nu}w\pare{v-u}-w\pare{\partial_{\nu}v-\condcor\partial_{\nu}u} \,ds
	\end{equation}
	for any solution $w$ to
	\begin{equation}\label{eq:PDEomega}
	\nabla\cdot \condcor \nabla w + k^2\rho w=0,\quad \mbox{in $\mathcal{C}_{\epsilon}$}.
	\end{equation}
\end{lem}

Denote the vector field $\vec{b}=\vec{b}_{\condcor}$ as
\begin{equation*}
\vec{b}=\vec{b}_{\condcor}(x):=-\frac{\nabla\condcor^{-1/2}}{\condcor^{-1/2}}
=\frac{\nabla\condcor^{1/2}}{\condcor^{1/2}}
=2\frac{\nabla\condcor}{\condcor}
=-2\frac{\nabla\condcor^{-1}}{\condcor^{-1}}
=2\nabla \ln \condcor.
\end{equation*} 
\begin{lem}\label{lem:RapidDecay}
	Given  $\condcor\in W^{1,\infty}$ satisfying \eqref{eq:gammEllipCor}, let $u,v\in H^1(\mathcal{C}_{\epsilon})$ satisfy \eqref{eq:TEVPlocal} and let $w$ be a solution to \eqref{eq:PDEomega} of the form \eqref{eq:CGO} with $r\in H^1(\mathcal{C}_{\epsilon})$ and $d\in \Ss^{n-1}$ satisfying \eqref{eq:dProperty}. 
	Then one has
	\begin{equation}\label{eq:DecayRapi}
	\abs{\int_{\mathcal{C}_{\epsilon}}\pare{\condcor-1} \nabla v\cdot\nabla w-k^2(\rho-1) v w \,dx}
	= O(\tau e^{-\delta\tau \epsilon}),
	\end{equation}
	for $\tau$ sufficiently large.
\end{lem}
\begin{proof}
	It is observed that
	\begin{equation*}
	\abs{e^{\eta\cdot x}}\le e^{-\tau d\cdot x}\le e^{-\tau \delta |x|},
	\quad \mbox{for any $x\in \mathcal{C}\cap \overline{B_{\epsilon}}$}. 
	\end{equation*}
	Hence for $w$ of the form \eqref{eq:CGO} we have
	\begin{equation*}
	\begin{split}
	\abs{\int_{\mathcal{K}_{\epsilon}}w\pare{\partial_{\nu}v-\condcor\partial_{\nu}u}}
	=&\abs{\int_{\mathcal{K}_{\epsilon}}\condcor^{-1/2}(1+r(x))\pare{\partial_{\nu}v-\condcor\partial_{\nu}u}e^{\eta\cdot x}}
	\\\le&  e^{-\epsilon  \delta \tau }\int_{\mathcal{K}_{\epsilon}}\abs{1+r(x)}\pare{\lambda_1^{-1/2}\abs{\partial_{\nu}v}+\lambda_2^{1/2}\abs{\partial_{\nu}u}} 
	\\\lesssim & e^{-\epsilon  \delta \tau },
	\end{split}	
	\end{equation*}
	where $\lambda_1$ and $\lambda_2$ are the constants in \eqref{eq:gammEllipCor}.
	Notice that
	\begin{equation*}
	\nabla \pare{\condcor^{-1/2}e^{\eta\cdot x}} 
	=(\eta-\vec{b})\condcor^{-1/2}e^{\eta\cdot x}.
	\end{equation*}
	Then we have
	\begin{equation}\label{eq:gradw}
	\nabla w = \condcor^{-1/2}\Pare{\nabla r +\pare{1+r}(\eta-\vec{b})}e^{\eta\cdot x},
	\end{equation}
	and hence
	\begin{equation*}
	\partial_{\nu} w =\condcor^{-1/2}\Pare{\partial_{\nu} r +\pare{1+r}(\eta-\vec{b})\cdot\nu}e^{\eta\cdot x}.
	\end{equation*}
	Therefore we have
	\begin{equation*}
	\begin{split}
	\abs{\int_{\mathcal{K}_{\epsilon}}\condcor\partial_{\nu}w\pare{v-u}}
	=&\abs{\int_{\mathcal{K}_{\epsilon}}\condcor^{1/2}\Pare{\partial_{\nu} r +\pare{1+r}(\eta-\vec{b})\cdot\nu}\pare{v-u}e^{\eta\cdot x}}
	\\\le&  \lambda_2^{1/2}\abs{\int_{\mathcal{K}_{\epsilon}}\Pare{\abs{\partial_{\nu} r}+2\abs{1+r}}\abs{v-u}}\tau e^{-\epsilon  \delta \tau }
	\\&+2\lambda_1^{1/2}\abs{\int_{\mathcal{K}_{\epsilon}}\partial_{\nu} \condcor \pare{1+r}\pare{v-u}} e^{-\epsilon  \delta \tau }
	\\\lesssim & \tau e^{-\epsilon \delta\tau}.
	\end{split}	
	\end{equation*}
	\noindent
	Now the proof can be completed by using the identity \eqref{eq:IntCor}.
\end{proof}

In the following, we will use repeatedly the estimate 
\begin{equation}\label{eq:EstiInt}
\int_{0}^{\epsilon}t^{b-1}e^{-\mu t}dt
=\Gamma(b)/\mu^{b} +o\pare{e^{-\epsilon \Re\mu /2}},\quad \mbox{as $\Re\mu\to\infty$},
\end{equation}
for any real number $b>0$ and any complex number $\mu$ satisfying $\Re\mu\gg 1$, 
where $\Gamma$ stands for the Gamma function.
We include the proof of \eqref{eq:EstiInt} for readers' convenience.
\begin{proof}[Proof of \eqref{eq:EstiInt}]
	Denote $\mu_1:=\Re\mu\gg 1$.
	Suppose that $\mu_1\ge 2(b-1)/s$. Then 
	\begin{equation*}
	t^{b-1}\le e^{\mu_1 t/2 }, \quad \mbox{for all $t\ge s$},
	\end{equation*}
	and hence
	\begin{equation*}
	\abs{\int_{\epsilon}^{\infty}t^{b-1}e^{-\mu t}dt}
	\le \int_{\epsilon}^{\infty}t^{b-1}e^{-\mu_1t }dt
	\le \int_{\epsilon}^{\infty}e^{-\mu_1 t/2}dt
	=2e^{-\mu_1 s/2}/\mu_1.
	\end{equation*}	
	Notice that
	\begin{equation*}
	\int_{0}^{\infty}t^{b-1}e^{-\mu t}dt
	=\mathcal{L}\{t^{b-1}\}(\mu)=\Gamma(b)/\mu^{b},
	\end{equation*}
	where $\mathcal{L}$ represents the Laplace transform and $\Gamma$ is the gamma function.
	Therefore, we have
	\begin{equation*}
	\begin{split}
	\int_{0}^{\epsilon}t^{b-1}e^{-\mu t}dt
	&=\int_{0}^{\infty}t^{b-1}e^{-\mu t}dt	-\int_{\epsilon}^{\infty}t^{b-1}e^{-\mu t}dt
	\\&=\Gamma(b)/\mu^{b} +o\pare{e^{-\mu_1 s/2}},\quad \mbox{as $\mu_1=\Re\mu\to\infty$}.
	\end{split}
	\end{equation*}
\end{proof}

\begin{lem}\label{lem:ContraV}
	Under the same notations and conditions as in Lemma~\ref{lem:RapidDecay}, suppose that there are constants $p,\hat{p}>1$ and $\sigma_0>0$ such that the residual $r$  in the form \eqref{eq:CGO} of $w$ satisfies 
	\begin{equation}\label{eq:rDecay}
	\|r\|_{L^p(\mathcal{C}_{\epsilon})}=O(\tau^{-n/p-\sigma_0}) \quad\text{and}\quad
	\|\nabla r\|_{L^{\hat{p}}(\mathcal{C}_{\epsilon})}=O(\tau^{1-n/\hat{p}-\sigma_0})
	\end{equation} 
	for $\tau$ sufficiently large.
	Assume further that there exist constants $\alpha,\beta, \sigma \in\RR$, with $\sigma>0$, $\alpha\neq -1$, and functions $\condcor_\beta, V\in L^{\infty}(\mathcal{K})$, which are all independent of $\tau$, satisfying
	\begin{equation}\label{eq:ExpanG}
	(\gamma(x)-1)\gamma^{-1/2}(x)= \condcor_{\beta}(\hat{x})|x|^{\beta} +O(|x|^{\beta+\sigma}),
	\end{equation}
	and
	\begin{equation}\label{eq:ExpanV}
	\nabla v(x)=V(\hat{x})|x|^{\alpha}
	+O(|x|^{\alpha+\sigma}),
	\end{equation}
	for all $x\in \mathcal{C}_{\epsilon}$.
	Then one must have
	\begin{equation}\label{eq:IntVanishV2}
		\begin{split}
			&\abs{\int_{\mathcal{C}_{\epsilon}}
				\condcor_{\beta}(\hat{x})V(\hat{x})\cdot \eta|x|^{\alpha+\beta}e^{\eta\cdot x}dx}
			= \|\condcor_{\beta}\|_{L^{\infty}(\mathcal{K})}\|V\|_{L^{\infty}(\mathcal{K})}\,O\big(\frac{1}{\tau^{n+\beta+\alpha-1}}\big),
		\end{split}
	\end{equation}
	and
	\begin{equation}\label{eq:IntVanishV}
	\begin{split}
	&\abs{\int_{\mathcal{C}_{\epsilon}}
		\left[\pare{\condcor(x)-1} \nabla v(x)\cdot\nabla w(x)
		- \condcor_{\beta}(\hat{x})V(\hat{x})\cdot \eta|x|^{\alpha+\beta}e^{\eta\cdot x}\right]dx}
	\\&= \|\condcor_{\beta}\|_{L^{\infty}(\mathcal{K})}\|V\|_{L^{\infty}(\mathcal{K})}\,O\big(\frac{1}{\tau^{n+\beta+\alpha}}\big)+ O\big(\frac{1}{\tau^{n+\beta+\alpha-1+\sigma}}\big),
	\end{split}
	\end{equation}
	as $\tau\to\infty$, for any $d\in\mathcal{K}'_{\delta}$ with $\delta>0$.
\end{lem}
\begin{proof}
	Recalling \eqref{eq:gradw} from before we have
	\begin{equation*}
	\begin{split}
	&\int_{\mathcal{C}_{\epsilon}}\pare{\condcor-1} \nabla v\cdot\nabla w\,dx=\int_{\mathcal{C}_{\epsilon}}\condcor^{-1/2}\pare{\condcor-1} \nabla v\cdot \pare{\eta-\vec{b}+\Pare{\nabla+\eta-\vec{b}\,}r}e^{\eta\cdot x}dx
	\end{split}
	\end{equation*}	
	Using \eqref{eq:ExpanG} and \eqref{eq:ExpanV} we observe that 
	\begin{equation*}
	\pare{\condcor-1}\condcor^{-1/2}\nabla v
	=\condcor_{\beta}(\hat{x})V(\hat{x})|x|^{\alpha+\beta}
	+O\pare{|x|^{\alpha+\beta+\sigma}}.
	\end{equation*} 
	Thus we are able to split the integral in \eqref{eq:IntVanishV} as
	\begin{equation}\label{eq:IntSplitV}
	\int_{\mathcal{C}_{\epsilon}}
		\pare{\condcor-1} \nabla v\cdot\nabla w\,dx
		-I_{10}
	=\sum_{j=1}^{3}I_{1j},
	\end{equation}
	where the integrals $I_{1j}$, $j=0,\ldots,3$, are defined by
	\begin{equation*}
	I_{10}:=\int_{\mathcal{C}_{\epsilon}}
	\condcor_{\beta}(\hat{x})V(\hat{x})\cdot \eta|x|^{\alpha+\beta}e^{\eta\cdot x}dx,
	\end{equation*}
	\begin{equation*}
	I_{11}:=-\int_{\mathcal{C}_{\epsilon}}
	\condcor_{\beta}(\hat{x})|x|^{\alpha+\beta}
	V(\hat{x})\cdot \vec{b}(x)e^{\eta\cdot x}dx,
	\end{equation*}
	\begin{equation*}
	I_{12}:=\int_{\mathcal{C}_{\epsilon}}\condcor_{\beta}(\hat{x})|x|^{\alpha+\beta}V(\hat{x})\cdot \pare{\nabla+\eta-\vec{b}(x)}r e^{\eta\cdot x}dx,
	\end{equation*}
	and
	\begin{equation*}
	I_{13}:=\int_{\mathcal{C}_{\epsilon}}
	O\pare{|x|^{\alpha+\beta+\sigma}}\cdot \pare{\eta-\vec{b}(x)+\Pare{\nabla+\eta-\vec{b}(x)}r}e^{\eta\cdot x}dx,
	\end{equation*}
	which can all be viewed as functions of $\tau$.

	Notice that
	\begin{equation*}
	|I_{10}|\le \tau \|\condcor_{\beta} V\|_{L^{\infty}(\mathcal{K})}
	\int_{\mathcal{C}_{\epsilon}}e^{-\tau d\cdot x}|x|^{\alpha+\beta} \,dx.
	\end{equation*}
	By applying \eqref{eq:EstiInt} and the property \eqref{eq:dProperty} for $d\in\mathcal{K}'_{\delta}$ we obtain that
	\begin{equation}\label{eq:est3}
	\begin{split}
	|I_{10}|\lesssim& \,\tau\,\|\condcor_{\beta}V\|_{L^{\infty}(\mathcal{K})}\int_{0}^{\epsilon}\int_{\mathcal{K}}t^{n+\alpha+\beta-1}e^{- t \tau d\cdot \hat{x} } \,d\sigma(\hat{x})dt
	\\&\lesssim \,\tau\,\|\condcor_{\beta}V\|_{L^{\infty}(\mathcal{K})}\int_{0}^{\epsilon}t^{n+\alpha+\beta-1}e^{- t \tau \delta } \,dt
	\lesssim  \frac{\|\condcor_{\beta}V\|_{L^{\infty}(\mathcal{K})}}{\tau^{n+\alpha+\beta-1}}.
	\end{split}
	\end{equation}
	Using the same arguments and recalling \eqref{eq:gammEllip} we also observe that
	\begin{equation*}
	\begin{split}
	|I_{11}|
	\le& \lambda_1^{-1}\|\nabla\condcor\|_{L^{\infty}(\mathcal{C}_{\epsilon})}\|\condcor_{\beta} V\|_{L^{\infty}(\mathcal{K})}
	\int_{\mathcal{C}_{\epsilon}}e^{-\tau d\cdot x}|x|^{\alpha+\beta} \,dx
	\lesssim  \frac{\|\condcor_{\beta}V\|_{L^{\infty}(\mathcal{K})}}{\tau^{n+\alpha+\beta}}.
	\end{split}
	\end{equation*}
	For the estimate concerning $I_{12}$, we first have that
	\begin{equation*}
	\begin{split}
	\abs{\int_{\mathcal{C}_{\epsilon}}re^{\eta\cdot x}\condcor_{\beta}(\hat{x})|x|^{\alpha+\beta}V(\hat{x})\cdot \pare{\eta-\vec{b}(x)} dx}
	&\lesssim \tau \|\condcor_{\beta}V\|_{L^{\infty}(\mathcal{K})} \int_{\mathcal{C}_{\epsilon}}e^{-\tau d\cdot x}|x|^{\alpha+\beta}  r \,dx
	\\&\hspace{-1in}\lesssim  \tau \|\condcor_{\beta}V\|_{L^{\infty}(\mathcal{K})} \|r\|_{L^p(\mathcal{C}_{\epsilon})}\pare{\int_{\mathcal{C}_{\epsilon}}e^{-p' \tau d\cdot x}|x|^{p'(\alpha+\beta)}\,dx}^{1/p'},
	\end{split}
	\end{equation*}
	where $p'$ denotes the H\"{o}lder conjugate of $p$.
	In the same way as for \eqref{eq:est3}, we can derive that
	\begin{equation*}
	\int_{\mathcal{C}_{\epsilon}}e^{-p' \tau d\cdot x}|x|^{p'(\alpha+\beta)}\,dx
	\le \sigma(\mathcal{K})\int_{0}^{\epsilon}t^{n+p'(\alpha+\beta)-1}e^{- t p'\tau \delta } \,dt
	\lesssim  \frac{1}{\tau^{n+p'(\alpha+\beta)}},
	\end{equation*}
	and hence that
	\begin{equation*}
	\abs{\int_{\mathcal{C}_{\epsilon}}re^{\eta\cdot x}\condcor_{\beta}(\hat{x})|x|^{\alpha+\beta}V(\hat{x})\cdot \pare{\eta-\vec{b}(x)} dx}
	\lesssim \frac{\|\condcor_{\beta}V\|_{L^{\infty}(\mathcal{K})}\|r\|_{L^p(\mathcal{C}_{\epsilon})}}{\tau^{n/p'+\alpha+\beta-1}}
	=\frac{\|\condcor_{\beta}V\|_{L^{\infty}(\mathcal{K})}\|r\|_{L^p(\mathcal{C}_{\epsilon})}}{\tau^{n+\alpha+\beta-1-n/p}}.
	\end{equation*}
	Similarly, we can obtain the following estimate for the rest part of the integral $I_{12}$,
	\begin{equation*}
	\begin{split}
	\abs{\int_{\mathcal{C}_{\epsilon}}e^{\eta\cdot x}\condcor_{\beta}(\hat{x})|x|^{\alpha+\beta}V(\hat{x})\cdot \nabla r\, dx}
	&\lesssim \|\condcor_{\beta}V\|_{L^{\infty}(\mathcal{K})} \|\nabla r\|_{L^{\hat{p}}(\mathcal{C}_{\epsilon})}\pare{\int_{\mathcal{C}_{\epsilon}}e^{-\hat{p}' \tau d\cdot x}|x|^{\hat{p}'(\alpha+\beta)}\,dx}^{1/\hat{p}'}
	\\&\lesssim \frac{ \|\condcor_{\beta}V\|_{L^{\infty}(\mathcal{K})}\|\nabla r\|_{L^{\hat{p}}(\mathcal{C}_{\epsilon})}}{\tau^{n+\alpha+\beta-n/\hat{p}}}.
	\end{split}
	\end{equation*}
	Therefore, the condition \eqref{eq:rDecay} implies
	\begin{equation*}
	|I_{12}|= \|\condcor_{\beta}V\|_{L^{\infty}(\mathcal{K})}\, O\big(\frac{1}{\tau^{n+\alpha+\beta-1+\sigma_0}}\big).
	\end{equation*}
	As for the integral $I_{13}$, analogous to the estimates for $I_{11}$ and $I_{12}$ one has 
	\begin{equation*}
	\begin{split}
	\abs{\int_{\mathcal{C}_{\epsilon}}
		e^{\eta\cdot x}O\pare{|x|^{\alpha+\beta+\sigma}}\cdot \pare{-\vec{b}+\Pare{\nabla+\eta-\vec{b}\, }r}dx}
	= O\big(\frac{1}{\tau^{n+\alpha+\beta+\sigma}}\big)
	+O\big(\frac{1}{\tau^{n+\alpha+\beta-1+\sigma+\sigma_0}}\big).
	\end{split}
	\end{equation*}
	Applying the argument in \eqref{eq:est3} for the rest of $I_{13}$ we obtain that
	\begin{equation*}
	\begin{split}
	\abs{\int_{\mathcal{C}_{\epsilon}}
		e^{\eta\cdot x}\eta\cdot O\pare{|x|^{\alpha+\beta+\sigma}} dx}
	&\lesssim  \tau \int_{\mathcal{C}_{\epsilon}}e^{-\tau d\cdot x}|x|^{\alpha+\beta+\sigma} dx
	\lesssim \frac{1}{\tau^{n+\alpha+\beta+\sigma-1}}.
	\end{split}
	\end{equation*}
	Hence we have 
	\begin{equation*}
	|I_{13}|=O\big(\frac{1}{\tau^{n+\alpha+\beta+\sigma-1}}\big).
	\end{equation*}
	
	The proof is complete.
\end{proof}
\begin{lem}\label{lem:Contrav}
	Adopting the same notations and assumptions as in Lemma~\ref{lem:ContraV},  assume further that there exist constants $\alpha_0, \beta_0 \in\RR$ and functions $\rho_0, \tilde{v}\in L^{\infty}(\mathcal{K})$, which are all independent of $\tau$, satisfying
	\begin{equation}\label{eq:ExpanR}
	(\rho(x)-1)\gamma^{-1/2}(x)=\rho_0(\hat{x})|x|^{\beta_0} +O(|x|^{\beta_0+\sigma}),
	\end{equation} 
	and
	\begin{equation}\label{eq:Expanv0}
	v(x)=\tilde{v}(\hat{x})|x|^{\alpha_0}+O(|x|^{\alpha_0+\sigma}),
	\end{equation}
	with some  $\delta>0$, for all $x\in \mathcal{C}_{\epsilon}$.
	Then one must have
	\begin{equation}\label{eq:IntVanishv2}
	\abs{\int_{\mathcal{C}_{\epsilon}} \rho_0(\hat x)\tilde{v}(\hat{x})|x|^{\alpha_0+\beta_0}e^{\eta\cdot x}
		dx} =\|\rho_0\|_{L^{\infty}(\mathcal{K})}\|\tilde{v}\|_{L^{\infty}(\mathcal{K})}\,O\big(\frac{1}{\tau^{n+\beta_0+\alpha_0}}\big),
	\end{equation}	
	and
	\begin{equation}\label{eq:IntVanishv}
	\begin{split}
	&\abs{\int_{\mathcal{C}_{\epsilon}} \left[(\rho-1) v w
		-\rho_0(\hat x)\tilde{v}(\hat{x})|x|^{\alpha_0+\beta_0}e^{\eta\cdot x}
		\right]dx} \\=&\|\rho_0\|_{L^{\infty}(\mathcal{K})}\|\tilde{v}\|_{L^{\infty}(\mathcal{K})}\,O\big(\frac{1}{\tau^{n+\beta_0+\alpha_0+\sigma_0}}\big)+O\big(\frac{1}{\tau^{n+\beta_0+\alpha_0+\sigma}}\big),
	\end{split}	
	\end{equation}
	as $\tau\to\infty$, for any $d\in\mathcal{K}'_{\delta}$ with $\delta>0$.
\end{lem}
\begin{proof}
	We first observe from \eqref{eq:ExpanR} and \eqref{eq:Expanv0} that
	\begin{equation*}
	(\rho-1)\gamma^{-1/2}v=\rho_0(\hat{x})\tilde{v}(\hat{x})|x|^{\alpha_0+\beta_0} +O(|x|^{\alpha_0+\beta_0+\sigma}).
	\end{equation*}
	Substituting the form \eqref{eq:CGO} of $w$ we can split the integral 
	\begin{equation}\label{eq:IntSplitv}
	\int_{\mathcal{C}_{\epsilon}} (\rho-1) v w\,dx-I_{20}
	=I_{21}+I_{22},
	\end{equation}
	where the integrals $I_{20}$, $I_{21}$ and $I_{22}$ are defined by
	\begin{equation*}
	I_{20}:=\int_{\mathcal{C}_{\epsilon}} \rho_0(\hat x)\tilde{v}(\hat{x})|x|^{\alpha_0+\beta_0}e^{\eta\cdot x}
	dx,
	\end{equation*}
	\begin{equation*}
	I_{21}:=\int_{\mathcal{C}_{\epsilon}} \rho_0(\hat x)\tilde{v}(\hat{x})|x|^{\alpha_0+\beta_0}r(x)e^{\eta\cdot x} dx,
	\end{equation*}
	and
	\begin{equation*}
	I_{22}:=\int_{\mathcal{C}_{\epsilon}}O\pare{|x|^{\alpha_0+\beta_0+\sigma}}(1+r)e^{\eta\cdot x}dx,
	\end{equation*}
	which can be regarded as functions of $\tau$.	
	Applying similar arguments as in the proof of Lemma~\ref{lem:ContraV} and using \eqref{eq:EstiInt} we have
	\begin{equation*}
	\begin{split}
	|I_{20}|\lesssim&  \|\rho_0\tilde{v}\|_{L^{\infty}(\mathcal{K})}
	\int_{\mathcal{C}_{\epsilon}}e^{-\tau d\cdot x}|x|^{\alpha_0+\beta_0} \,dx
	\\&  = \|\rho_0\tilde{v}\|_{L^{\infty}(\mathcal{K})}\int_{0}^{\epsilon}\int_{\mathcal{K}}t^{n+\alpha_0+\beta_0-1}e^{- t \tau d\cdot \hat{x} } \,d\sigma(\hat{x})dt
	\\&\lesssim \|\rho_0\tilde{v}\|_{L^{\infty}(\mathcal{K})}
	\int_{0}^{\epsilon}t^{n+\alpha_0+\beta_0-1}e^{- t \tau \delta } \,dt
	\lesssim  \frac{\|\rho_0\tilde{v}\|_{L^{\infty}(\mathcal{K})}}{\tau^{n+\alpha_0+\beta_0}}.
	\end{split}
	\end{equation*}
	Making use of \eqref{eq:rDecay} we can derive
	\begin{equation*}
	\begin{split}
	|I_{21}|&\lesssim  \|\rho_0\tilde{v}\|_{L^{\infty}(\mathcal{K})}
	\int_{\mathcal{C}_{\epsilon}}re^{-\tau d\cdot x}|x|^{\alpha_0+\beta_0} \,dx
	\\&\lesssim  \|\rho_0\tilde{v}\|_{L^{\infty}(\mathcal{K})}\|r\|_{L^p(\mathcal{C}_{\epsilon})}\pare{\int_{\mathcal{C}_{\epsilon}}e^{-p' \tau d\cdot x}|x|^{p'(\beta_0+\alpha_0)}\,dx}^{1/p'}
	\\&\lesssim   \|\rho_0\tilde{v}\|_{L^{\infty}(\mathcal{K})}\frac{\|r\|_{L^p(\mathcal{C}_{\epsilon})}}{\tau^{n/p'+\beta_0+\alpha_0}}= \|\rho_0\tilde{v}\|_{L^{\infty}(\mathcal{K})}\frac{\|r\|_{L^p(\mathcal{C}_{\epsilon})}}{\tau^{n+\beta_0+\alpha_0-n/p}}
	\lesssim   \frac{\|\rho_0\tilde{v}\|_{L^{\infty}(\mathcal{K})}}{\tau^{n+\beta_0+\alpha_0+\sigma_0}}.
	\end{split}
	\end{equation*}
	Now the following estimate for $I_{22}$ can be obtained analogously from those for $I_{20}$ and $I_{21}$,
	\begin{equation*}
	|I_{22}|\lesssim   \frac{1}{\tau^{n+\beta_0+\alpha_0+\sigma}}
	+\frac{1}{\tau^{n+\beta_0+\alpha_0+\sigma+\sigma_0}}.
	\end{equation*}
	The proof is complete.
\end{proof}

The following result provides an estimate of an integral over $\mathcal{C}_{\epsilon}$ involving the solution $v$ of the problem \eqref{eq:TEVPlocal}. 
\begin{prop}\label{prop:Contra}
	Under the same notations and assumptions as in Lemmas~\ref{lem:ContraV} and \ref{lem:Contrav}, one must have
	\begin{equation}\label{eq:IntVanish0}
	\begin{split}
	&\abs{\int_{\mathcal{C}_{\epsilon}}\condcor_{\beta}(\hat{x})V(\hat{x})\cdot \eta|x|^{\alpha+\beta}e^{\eta\cdot x} dx} \\=&k^2\|\rho_0\|_{L^{\infty}(\mathcal{K})}\|\tilde{v}\|_{L^{\infty}(\mathcal{K})}\,O\big(\frac{1}{\tau^{n+\beta_0+\alpha_0}}\big)+k^2O\big(\frac{1}{\tau^{n+\beta_0+\alpha_0+\sigma}}\big)
	\\&+\|\condcor_{\beta}\|_{L^{\infty}(\mathcal{K})}\|V\|_{L^{\infty}(\mathcal{K})}\,O\big(\frac{1}{\tau^{n+\beta+\alpha}}\big)+ O\big(\frac{1}{\tau^{n+\beta+\alpha-1+\sigma}}\big),
	\end{split}	
	\end{equation}
	\begin{equation}\label{eq:IntVanishpho}
	\begin{split}
	&k^2\abs{\int_{\mathcal{C}_{\epsilon}}e^{\eta\cdot x} \tilde{v}(\hat x)\rho_0(\hat x)|x|^{\alpha_0+\beta_0}dx} \\=&k^2\|\rho_0\|_{L^{\infty}(\mathcal{K})}\|\tilde{v}\|_{L^{\infty}(\mathcal{K})}\,O\big(\frac{1}{\tau^{n+\beta_0+\alpha_0+\sigma_0}}\big)+k^2O\big(\frac{1}{\tau^{n+\beta_0+\alpha_0+\sigma}}\big)
	\\&+\|\condcor_{\beta}\|_{L^{\infty}(\mathcal{K})}\|V\|_{L^{\infty}(\mathcal{K})}\,O\big(\frac{1}{\tau^{n+\beta+\alpha-1}}\big)
	+ O\big(\frac{1}{\tau^{n+\beta+\alpha-1+\sigma}}\big),
	\end{split}	
	\end{equation}
	and
	\begin{equation}\label{eq:IntVanish}
	\begin{split}
	&\abs{\int_{\mathcal{C}_{\epsilon}}e^{\eta\cdot x} \left[
	-k^2\tilde{v}(\hat x)\rho_0(\hat x)|x|^{\alpha_0+\beta_0}
	+\condcor_{\beta}(\hat{x})V(\hat{x})\cdot \eta|x|^{\alpha+\beta}\right]dx} \\=&k^2\|\rho_0\|_{L^{\infty}(\mathcal{K})}\|\tilde{v}\|_{L^{\infty}(\mathcal{K})}\,O\big(\frac{1}{\tau^{n+\beta_0+\alpha_0+\sigma_0}}\big)+k^2O\big(\frac{1}{\tau^{n+\beta_0+\alpha_0+\sigma}}\big)
	\\&+\|\condcor_{\beta}\|_{L^{\infty}(\mathcal{K})}\|V\|_{L^{\infty}(\mathcal{K})}\,O\big(\frac{1}{\tau^{n+\beta+\alpha}}\big)+ O\big(\frac{1}{\tau^{n+\beta+\alpha-1+\sigma}}\big),
	\end{split}	
	\end{equation}
	as $\tau\to\infty$, for any $d\in\mathcal{K}'_{\delta}$ with some $\delta>0$.
\end{prop}
\begin{proof}
	This is a direct consequence of Lemmas~\ref{lem:RapidDecay}, \ref{lem:ContraV} and \ref{lem:Contrav}, more precisely, by rewriting
	\begin{equation*}
	\begin{split}
	&\int_{\mathcal{C}_{\epsilon}}e^{\eta\cdot x} \left[
	-k^2\tilde{v}(\hat x)\rho_0(\hat x)|x|^{\alpha_0+\beta_0}
	+\condcor_{\beta}(\hat{x})V(\hat{x})\cdot \eta|x|^{\alpha+\beta}\right]dx
	\\=& \,k^2\int_{\mathcal{C}_{\epsilon}} \left[(\rho-1) v w
	-\rho_0(\hat x)\tilde{v}(\hat{x})|x|^{\alpha_0+\beta_0}e^{\eta\cdot x}
	\right]dx
	\\&-\int_{\mathcal{C}_{\epsilon}}
	\left[\pare{\condcor-1} \nabla v\cdot\nabla w
	- \condcor_{\beta}(\hat{x})V(\hat{x})\cdot \eta|x|^{\alpha+\beta}e^{\eta\cdot x}\right]dx
	\\&+\int_{\mathcal{C}_{\epsilon}}\pare{\condcor-1} \nabla v\cdot\nabla w-k^2(\rho-1) v w \,dx,
	\end{split}
	\end{equation*}
	and using the estimates \eqref{eq:DecayRapi}, \eqref{eq:IntVanishV}, \eqref{eq:IntVanishv2}, \eqref{eq:IntVanishV2} and \eqref{eq:IntVanishv}.
\end{proof}
\begin{rem}
	Under sufficient regularity, the term $\condcor_{\beta}(\hat{x})|x|^{\beta}$ in \eqref{eq:ExpanG} can be regarded as the first non-zero term, which is a homogeneous polynomial in this case, of the Taylor expansion for $\condcor^{-1/2}(\condcor-1)$ around $x=0$.
	The same situation is true for  \eqref{eq:ExpanV}, \eqref{eq:ExpanR} and \eqref{eq:Expanv0}, concerning $V$,  $\rho$ and $v$, respectively.
\end{rem}
\begin{rem}
	Later on, we will present some situations or conditions which would yield a contradiction of \eqref{eq:IntVanish}. As a consequence, we will be able to characterize the non-scattering property as well as some behavior of transmission eigenfunctions under certain circumstances.
\end{rem}
\begin{cor}\label{Cor:ContraSim}
	Under the same notations and assumptions as in Proposition~\ref{prop:Contra}, assume further that $\beta=\beta_0=0$, $N:=\alpha\ge 0$  an integer and both $\gamma_{\beta}=\gamma_0$ and $\rho_0$ are constants. 
	If $\alpha_0\ge N=\alpha$, then one has
	\begin{equation}\label{eq:IntVanish3}
	\begin{split}
	&\condcor_0\abs{\int_{\mathcal{C}_{\epsilon}} V(\hat{x})\cdot \eta|x|^{N}e^{\eta\cdot x} \,dx } =o\big(\tau^{1-n-N}\big).
	\end{split}
	\end{equation}
	Otherwise if $\alpha_0=N-1$, then 
	\begin{equation}\label{eq:IntVanish4}
	\begin{split}
	&\abs{\condcor_0\int_{\mathcal{C}_{\epsilon}} V(\hat{x})\cdot \eta|x|^{N}e^{\eta\cdot x} \,dx-k^2\rho_0\int_{\mathcal{C}_{\epsilon}}e^{\eta\cdot x}\tilde{v}(\hat x)|x|^{N-1}
		dx}=o\big(\tau^{1-n-N}\big).
	\end{split}	
	\end{equation}
\end{cor}
\begin{rem}
	The constants $\gamma_0$ and $\rho_0$ in Corollary~\ref{Cor:ContraSim} can be viewed as the contrast of the coefficients $\gamma(x)$ and $\rho(x)$ comparing to the constant $1$. In fact, under sufficient smoothness, one has
	\begin{equation*}
	\gamma_0=\pare{\gamma(0)-1}\gamma^{-1/2}(0)\quad\text{and}\quad 
	\rho_0=\pare{\rho(0)-1}\gamma^{-1/2}(0).
	\end{equation*}
\end{rem}

\medskip

\noindent
The above corollary highlights the complicated interplay of the behavior  near the corner of the contrasts in $\gamma$ and $\rho$, as well as of the fields $v$ and $\nabla v$,  in deciding  whether  (\ref{eq:IntVanishpho}) or (\ref{eq:IntVanish}) or both  are the dominating terms in the asymptotic expansions.
\section{Scattering by inhomogeneities with corners in 2D}\label{non-scat}

We revisit the problem \eqref{eq:TEVPlocal}, or \eqref{eq:ITPpde}, in this section. 
We prove in space dimension $n=2$ that, when certain conditions are satisfied, the asymptotics \eqref{eq:IntVanish} can not hold true unless the vector field $V$ is trivial.  As a consequence, we are able to derive some results concerning the ``never trivial'' scattering property of an inhomogeneous media in dimension two whose contrast in the main operator or/and the lower order term has a corner in its support.

\subsection{Preliminaries}
We first introduce some preliminary results. The first one is a standard result, see e.g. \cite{coltonkress}.
\begin{lem}
	If $v$ is a solution to the Helmholtz equation 
	\begin{equation}\label{eq:Helm}
	\Delta v +k^2 v=0.
	\end{equation}
	in an open domain in $\RR^n$, then $v$ is real analytic in that domain. 
\end{lem}

\begin{lem}\label{lem:TaylorvV}
	Let $v$ be defined in a neighborhood of a point $x_0\in \RR^n$ and satisfy the Helmholtz equation \eqref{eq:Helm}.	
	Write the Taylor series of $v$ and $\nabla v$ around $x_0$ as
	\begin{equation*}
	v=\sum_{j=N_0}^{\infty} v_{j}\quad \text{and} \quad \nabla v=\sum_{j=N}^{\infty} V_{j}\,,
	\end{equation*}
	where $v_{j}$ and $V_{j}$ are homogeneous (vectorial) polynomials of $(x-x_0)$ with degree $j$ for each $j\in \mathbb{N}$, $v_{N_0}$ and $V_{N}$ are not identically zero, and $N_0,N\in\mathbb{N}$. 
	Then the following are true, in the neighborhood where both of the Taylor series converges, 
	\begin{enumerate}
		\item The vector field $V_j$ is curl free for each $j$.
		\item There holds $N\le N_0\le N+1$, except for the case that $N_0=0$ and $N=1$. In the latter case, one must have in addition, $v_1=\nabla\cdot V_2=\Delta v_3\equiv 0$.
		\item $V_{N}$ is divergence free if $N\neq 1$, and $\nabla\cdot V_{N}=-k^2v_0$ when $N=1$.
		\item The polynomials $v_{N_0}$, $v_{N_0+1}$, $V_{N}$ and $V_{N+1}$ are harmonic.
	\end{enumerate}
\end{lem}
\begin{proof}
	Notice that 
	\begin{equation*}
	\nabla v = \sum_{j=N_0}^{\infty} \nabla v_{j},
	\end{equation*}
	with $\nabla v_{j}$ homogeneous vectorial polynomials of degree $j-1$, which is curl free, for each $j\ge 1$. Hence each $V_j$ is curl free, and we also observe $N_0-1\le N$.
	On the other hand from the Helmholtz equation we have
	\begin{equation*}
	-k^2v=\nabla\cdot\nabla v=\sum_{j=N}^{\infty} \nabla\cdot V_{j}\,.
	\end{equation*}
	Compare it to the original Taylor series of $v$ we obtain that $N-1\le N_0$.
	Now let us look at the case when $N_0=N-1\ge  0$. If $N\neq 1$, then $\nabla v_{N_0}=\nabla v_{N-1}$ is either identically zero or a homogeneous vectorial polynomial of degree $N-2\ge 0$. However, we know that the first nonzero term from the Taylor series of $\nabla v$ should be $V_{N}$. Therefore, we must have $\nabla v_{N-1}\equiv 0$, which implies that $v_{N-1}=v_{N_0}$ is a constant, namely, $N=1$.
	Next, we verify that $\nabla\cdot V_{N}=0$ when $N\neq 1$. If $N=0$, this is trivial since $V_{0}$ is a constant vector. For $N\ge 2$, we have that $\nabla\cdot V_{N}=-k^2v_{N-1}=0$, because we have shown that $v_{j}\equiv 0$ for all $j\le N-1$. 
	The last statement is known, see, \cite{BPS14,PSV17}. It can be seen directly by taking Laplacian on each term of the Taylor series and using the fact that both $v$ and $\nabla v$ solve the Helmholtz equation.
\end{proof}

We are now in a position to introduce an estimate which can be related to \eqref{eq:IntVanish}. 
In the following, we shall restrict ourselves only in dimension $n=2$. However, similar estimates and results are expected for dimension three or higher. Under this consideration, we still keep the notation $n$, instead of $2$, and specify $n=2$ when needed.

We define our local corner first.
Denote $\psi_0\in(0,\pi)$ as the aperture of a (convex) corner. Given positive constants $\epsilon$ and $\delta$, let $\mathcal{K}=\{(\cos\psi,\sin\psi);\ 0<\psi<\psi_0\}$, ${\mathcal C}$, ${\mathcal C}_\epsilon$ and $\mathcal{K}'_{\delta}$ be defined accordingly as in the beginning of  Section \ref{sec:corner}.
In particular, we remind here that $\mathcal{K}'_{\delta}$ is an open set of $\Ss^{n-1}$ where elements $d$ satisfy 
\eqref{eq:dProperty}.
\begin{lem}\label{lem:DecayExact}
	Let $n=2$, and let the complex vector $\eta$ be of the form \eqref{eq:eta} with $\tau>0$ and $d\in \mathcal{K}'_{\delta}$. 
	Given $N\in \mathbb{N}$, let $\tilde{V}=\tilde{V}(x)$ be the gradient of a homogeneous polynomial of degree $N+1$ which is harmonic. 
	Suppose that $\tilde{V}$ is not identically zero.
	Then one must have 
	\begin{equation}\label{eq:decayInt}
	\int_{\mathcal{C}_{\epsilon}}e^{\eta\cdot x} \tilde{V}\cdot \eta \,dx=C_0\tau^{1-n-N}+o\pare{\tau e^{-\epsilon \tau /2}},
	\end{equation}
	with a constant $C_0$ independent of $\tau$.
	Moreover, if $C_0$ is zero when taking both the two opposite directions of $d^{\perp}$ for fixed $d$, then one must have
	\begin{equation}\label{eq:Angle}
	\psi_0=\frac{2l\pi}{n+2N}=\frac{l\pi}{1+N}\in(0,\pi),\qquad   \mbox{i.e., } N=\frac{\pi}{\psi_0}l-1\in\mathbb{N},
	\end{equation}
	for some $l\in\mathbb{N}$.
\end{lem}
\begin{rem}\label{rem:DecayExactN0}
	When $N=0$, namely when $\tilde{V}$ is a constant vector, then $C_0\neq 0$ for corners of any angle unless $\tilde{V}$ is the zero vector. 
	If $N=1$, then \eqref{eq:Angle} implies that $\psi_0=\pi/2$ is the only case for $C_0=0$ with $\tilde{V}$ not identically zero.
\end{rem}

Before proving the above lemma, it is insightful to remark that the above exception of $\psi_0$ (which translates to a particular form of $u^{\In}$) is not an exception for the potential case $(\Delta+k^2\rho)u=0$, i.e when $\gamma=1$ (see i.e \cite{BPS14}). For our case concerning the operator $\nabla \cdot \gamma\nabla +k^2\rho$, even if we  replace $\tilde{V}\cdot \eta$ in \eqref{eq:decayInt} with $\tilde{V}\cdot \vec{p}$, with $\vec{p}\in \CC^2$ and $\vec{p}\cdot \eta=0$, would not exempt us from getting  exceptional $\psi_0$ that yield $C_0=0$ and non zero $\tilde{V}$.  In fact, such a vector $\vec{p}$ would satisfy
\begin{equation*}
\vec{p}= c_0(d^{\perp}-\im d)
=-\im c_0(d+\im d^{\perp})
=-\im c_0\eta/\tau.
\end{equation*}
This basically means that even a ``direct'' CGO solution for $\nabla w$ of the form $\nabla w =\gamma^{1/2}(\vec{p}+\vec{r})e^{\eta\cdot x}$ with $\vec{p}\in \CC^2$ satisfying $\vec{p}\cdot \eta=0$ might not be of help in improving the results.
\begin{proof}[Proof of Lemma \ref{lem:DecayExact}]
	It is known that $(x_1\pm \im x_2)^{N}$ form a base of all homogeneous harmonic polynomial of degree $N$, where $x_1$ and $x_2$ denote the Cartesian components of $x\in\RR^2$. Therefore, the vector field $\tilde{V}$ can be written as
	\begin{equation*}
	\begin{split}
	\tilde{V}(x)
	&=b_1\pare{\begin{array}{c}
		1  \\ \im
		\end{array}}(x_1+ \im x_2)^N
	+b_2\pare{\begin{array}{c}
		\im 	\\1 
		\end{array}}(x_1- \im x_2)^N
	\\&=b_1 \pare{\begin{array}{c}
		1\\\im
		\end{array}}|x|^N e^{\im N\psi}
	+b_2\pare{\begin{array}{c}
		\im\\1
		\end{array}}|x|^N e^{-\im N\psi},
	\end{split}
	\end{equation*}
	where we have adopted the parametrization as $\hat{x}=(\cos\psi,\sin\psi)^{T}$.
	Denote $d=(\cos\varphi,\sin\varphi)^{T}$. Taking $\varphi\mp \pi/2$ as the angular coordinate of $d^{\perp}$, then $d^{\perp}=\pm(\sin\varphi,-\cos\varphi)^{T}$  
	and
	\begin{equation*}
	d+\im d^{\perp}
	=\pare{\begin{array}{c}
		\cos\varphi\pm\im\sin\varphi\\
		\sin\varphi\mp\im\cos\varphi
		\end{array}}
	=\pare{\begin{array}{c}
		1\\
		\mp \im
		\end{array}}e^{\pm\im\varphi}.
	\end{equation*}
	Under these notations we have
	\begin{equation}\label{eq:proof1}
	\eta\cdot\hat{x} 
	=-\tau e^{\pm\im\varphi}\pare{\cos\psi\mp\im\sin\psi}
	=-\tau e^{\pm\im(\varphi-\psi)}.
	\end{equation}
	and, depending on the opposite direction choices of $d^{\perp}$,
	\begin{equation}\label{eq:proof4}
	\begin{split}
	-|x|^{-N}\tilde{V}(x)\cdot \eta \tau^{-1} 
	=&2b_1  e^{\im (N\psi+\varphi)}
	\quad\text{or}\quad
	-|x|^{-N}\tilde{V}(x)\cdot \eta \tau^{-1}=2\im b_2  e^{-\im (N\psi+\varphi)}.
	\end{split}
	\end{equation}
	It is observed that
	\begin{equation*}
	\int_{\mathcal{C}_{\epsilon}} |x|^{N}e^{\eta\cdot x} e^{\pm\im N\psi} \,dx
	=
	\int_{0}^{\epsilon}\int_{0}^{\psi_0}t^{N+n-1}e^{-t\tau e^{\pm\im(\varphi-\psi)}}e^{\pm\im N\psi}d \psi dt.	  
	\end{equation*}
	Applying the estimate \eqref{eq:EstiInt} yields
	\begin{equation}\label{eq:proof5}
	\begin{split}
	&\int_{\mathcal{C}_{\epsilon}} |x|^{N}e^{\eta\cdot x} e^{\pm\im N\psi} \,dx
	-o\pare{e^{-\epsilon \tau /2}}
	\\=&\frac{\Gamma(N+n)}{\tau^{N+n}}e^{\mp\im\varphi(N+n)}
	\int_{0}^{\psi_0} e^{\pm\im(N+n)\psi}e^{\pm\im N\psi }d \psi 
	=C_{\pm}\frac{\Gamma(N+n)}{\tau^{N+n}}
	,
	\end{split}	  
	\end{equation}
	with the constant
	\begin{equation*}
	\begin{split}
	C_{\pm}&:=e^{\mp\im(N+n)\varphi}\int_{0}^{\psi_0} e^{\pm\im(2N+n)\psi}d \psi 
	=	 \frac{\pm\im}{2N+ n}	\pare{1-e^{\pm\im(2N+ n)\psi_0}}e^{\mp\im(N+n)\varphi}.
	\end{split}
	\end{equation*}
	Therefore, we have from 
	\eqref{eq:proof4} and \eqref{eq:proof5} that
	\begin{equation*}
	\begin{split}
	\int_{\mathcal{C}_{\epsilon}}e^{\eta\cdot x} \tilde{V}\cdot \eta \,dx
	&=-2b_1 e^{\pm\im \varphi}\tau \int_{\mathcal{C}_{\epsilon}}|x|^{N}e^{\pm\im N\psi} e^{\eta\cdot x}dx 
	\\&=-2b_1C_{\pm}\frac{\Gamma(N+n)}{\tau^{N+n-1}}\,e^{\pm\im\varphi}
	+o\pare{\tau e^{-\epsilon \tau /2}},
	\end{split}	  
	\end{equation*}
	where the constant $b_1$ should be in fact $\im b_2$ when the $\mp$ is taken as the $+$ sign.
	We have now verified \eqref{eq:decayInt} with the constant 
	\begin{equation*}
	\begin{split}
	C_0&
	=-2\im b_1\frac{\Gamma(N+n)}{2N+ n}e^{-\im(N+n-1)\varphi}
	\pare{1-e^{\im(2N+ n)\psi_0}},
	\end{split}
	\end{equation*}
	if we take $\varphi-\pi/2$ as the angular of $d^{\perp}$,
	or if we take $\varphi+\pi/2$ as the angular of $d^{\perp}$
	\begin{equation*}
	\begin{split}
	C_0&=-2 b_2\frac{\Gamma(N+n)}{2N+ n}e^{\im(N+n-1)\varphi}
	\pare{1-e^{-\im(2N+ n)\psi_0}}.
	\end{split}
	\end{equation*}
	If $C_0=0$ for both cases, then one must have either 
	\begin{equation}\label{eq:Nexept}
	(2N+ n)\psi_0=2l\pi, \quad\mbox{for some $l\in\mathbb{N}$},
	\end{equation}
	or $b_1=b_2=0$.
	However, the latter cannot be true since we have assumed the non-triviality of $\tilde{V}$.	
\end{proof}

The following result is known, see \cite{PSV17}. It was first established in \cite{BPS14} for rectangular corners and $\gamma=1$. 
\begin{lem}\label{lem:DecayExactv}
	Let $n=2$, and let $\eta$ be of the form \eqref{eq:eta} with $\tau>0$ and $d\in \mathcal{K}'_{\delta}$. Let $v_{N_0}$ be a homogeneous polynomial of degree $N_0\in\mathbb{N}$ which is harmonic. Then there is a constant $C_{1,N_0}$, which depends on $d$ but not on $\tau$, such that
	\begin{equation}\label{eq:decayIntv}
	\int_{\mathcal{C}_{\epsilon}}v_{N_0}(x)e^{\eta\cdot x} \,dx=C_{1,N_0}\tau^{-n-N_0}+o\pare{\tau e^{-\epsilon \tau /2}}.
	\end{equation}
	Moreover, the constant $C_{1,N_0}=C_{1,N_0}(d)$ cannot be zero for all directions $d$ in any open subset of $\Ss^{n-1}$.
\end{lem}
The next result is a particular case of Lemma~\ref{lem:DecayExactv}, when $N_0=0$. We give  a proof for the sake of obtaining the explicit value of the constant $C_{1,N_0}$ in \eqref{eq:decayIntv}, which will be used later.
\begin{lem}\label{lem:DecayExact2}
	Under the same notations as in Lemma~\ref{lem:DecayExactv}, one has
	\begin{equation}\label{eq:decayInt2}
	\int_{\mathcal{C}_{\epsilon}}e^{\eta\cdot x} \,dx=C_1\tau^{-n}+o\pare{\tau e^{-\epsilon \tau /2}},
	\end{equation}
	with a constant $C_1\neq 0$ which is independent of $\tau$.
\end{lem}
\begin{proof}
	Applying \eqref{eq:proof1} and \eqref{eq:EstiInt} we have
	\begin{equation*}
	\begin{split}
	\int_{\mathcal{C}_{\epsilon}}e^{\eta\cdot x} \,dx
	&=\int_{\mathcal{C}_{\epsilon}}e^{-\tau|x| e^{\pm\im(\varphi-\psi)}}\,dx
	=\int_{0}^s\int_{0}^{\psi_0}t^{n-1}e^{-t \tau  e^{\pm\im(\varphi-\psi)} }d\psi dt
	\\&=\frac{\Gamma(n)}{\tau^{n}}
	\int_{0}^{\psi_0} e^{\mp\im n(\varphi-\psi)}\, d\psi
	+o\pare{ \tau e^{-\epsilon \tau /2}}.
	\end{split}	  
	\end{equation*}
	Therefore, we have derived \eqref{eq:decayInt2} with the constant
	\begin{equation*}
	C_1=\Gamma(n) e^{\mp\im n\varphi}
	\int_{0}^{\psi_0} e^{\pm\im n\psi}d \psi 
	=\pm\im \Gamma(n)/ne^{\mp\im n\varphi}\pare{1-e^{\pm\im n\psi_0}},
	\end{equation*}
	where the plus or minus signs depend on the choice of direction or angular, $\varphi\mp\pi/2$, of the unit vector $d^{\perp}$. 
\end{proof}
\begin{lem}\label{lem:DecayExact3}
	Let the dimension $n=2$, and Let $\tilde{V}=\tilde{V}(x)$ be a fixed homogeneous polynomial for $x$ of degree $N=1$ which is curl free and satisfies $\nabla\cdot \tilde{V}=-k^2v_0\ne 0$, where $v_0$ is a constant.
	Let the complex vector $\eta=\eta(\tau,d)$ be of the form \eqref{eq:eta} with $\tau>0$ and $d\in \mathcal{K}'_{\delta}$, and let $\gamma_{0}$ and $\rho_0$ be two constants with $\gamma_{0}\neq 0$. 	
	Then one has 
	\begin{equation}\label{eq:decayInt4}
	\int_{\mathcal{C}_{\epsilon}}e^{\eta\cdot x} \tilde{V}\cdot \eta \,dx=\widetilde{C}_0\tau^{-n}+o\pare{\tau e^{-\epsilon \tau /2}},
	\end{equation}
	and
	\begin{equation}\label{eq:decayInt3}
	\condcor_0\int_{\mathcal{C}_{\epsilon}} e^{\eta\cdot x}\tilde{V}\cdot \eta  \,dx 
	-k^2v_0\rho_0\int_{\mathcal{C}_{\epsilon}} e^{\eta\cdot x} \,dx=\widetilde{C}_1\tau^{-n}+o\pare{\tau e^{-\epsilon \tau /2}}, 
	\end{equation}
	with constants $\widetilde{C}_0$ and $\widetilde{C}_1$ independent of $\tau$ but possibly dependent on $d$.
	Moreover, {$\widetilde{C}_1=0$ for two opposite directions of $d^{\perp}$ if and only if $\psi_0=\pi/2$ and $\rho_0=\gamma_{0}$, or $\psi_0\neq\pi/2$ and $\tilde{V}$ takes the following form
	\begin{equation}\label{eq:TaylorVfirst}
	\tilde{V}(x)
	=\frac{k^2 v_0}{2}~\nabla \pare{\pare{1-\dfrac{\rho_0}{\gamma_0}}x_1x_2\tan\psi_0-\frac{1}{2}\frac{\rho_0}{\gamma_0}x_1^2-\frac{1}{2}\pare{2-\dfrac{\rho_0}{\gamma_0}}x_2^2}.
	\end{equation}}
\end{lem}
\begin{proof}
	We apply the parametrization $\hat{x}=(\cos\psi,\sin\psi)^{T}$ as before and write
	\begin{equation*}
	\tilde{V}(x)=\pare{\begin{array}{c}
		b_{11}  \\ b_{21} 
		\end{array}}(x_1+ \im x_2)
	+\pare{\begin{array}{c}
		b_{12} 	\\ b_{22}
		\end{array}}(x_1- \im x_2)
	=\pare{\begin{array}{c}
		b_{11}  \\ b_{21} 
		\end{array}}e^{\im\psi}
	+\pare{\begin{array}{c}
		b_{12} 	\\ b_{22}
		\end{array}}e^{-\im\psi}.
	\end{equation*}
	It is obtained from the curl and divergence condition that
	\begin{equation*}
	b_{21}-\im b_{11}=\im k^2 v_0/2\quad\text{and}\quad
	b_{12}-\im b_{22}=- k^2v_0/2.
	\end{equation*}
	We adopt the notations in the proof of Lemma~\ref{lem:DecayExact} for $d$, $d^{\perp}$ and $\eta$.
	Then
		\begin{equation*}
		\eta\cdot \tilde{V}({x})
		=-\tau |x| e^{\pm\im\varphi}\pare{b_{\pm}e^{\pm\im\psi}-k^2 v_0\, e^{\mp \im \psi}/2},
		\end{equation*}
		where $b_+=b_{11}-\im b_{21}$ and $b_-=b_{12}+\im b_{22}$.
		By straightforward computation we have
		\begin{equation*}
		\begin{split}
		\int_{\mathcal{C}_{\epsilon}} |x|\,e^{-\tau |x| e^{\pm \im (\varphi-\psi)}} e^{\pm\im \psi} \,dx
		=&\int_{0}^{\psi_0}\int_{0}^{\epsilon}r^2 \,e^{-\tau r e^{\pm\im (\varphi-\psi)}} e^{\pm\im \psi} \,drd\psi
		\\=&\,C_{\pm,1}\frac{\Gamma(1+n)}{\tau^{1+n}}
		+o\pare{e^{-\epsilon \tau /2}}
		,
		\end{split}
		\end{equation*}	
		and similarly
		\begin{equation*}
		\begin{split}
		\int_{\mathcal{C}_{\epsilon}} |x|\,e^{-\tau |x| e^{\pm \im (\varphi-\psi)}} e^{\mp\im \psi} \,dx
		=&\,C_{\pm,2}\frac{\Gamma(1+n)}{\tau^{1+n}}
		+o\pare{e^{-\epsilon \tau /2}}
		,
		\end{split}
		\end{equation*}	
		with the constants
		\begin{equation*}
		C_{\pm,1}=\pm \im e^{\mp \im (n+1)\varphi}\,\frac{1-e^{\pm\im (n+2)\psi_0}}{n+2}\qquad\text{and}\qquad
		C_{\pm,2}=\pm \im e^{\mp \im (n+1)\varphi}\,\frac{1-e^{\pm\im n\psi_0}}{n}.
		\end{equation*}
	Therefore, we have derived \eqref{eq:decayInt4} with the constant
	{\begin{equation*}
		\begin{split}
		\widetilde{C}_0
		&= \pm \im \Gamma(1+n)e^{\mp\im n\varphi}
		\pare{k^2v_0\frac{1-e^{\pm\im n\psi_0}}{2n}-b_{\pm}\frac{1-e^{\pm\im (n+2)\psi_0}}{n+2}}
		\\&=\pm \im e^{\mp\im 2\varphi}\pare{k^2v_0\pare{1-e^{\pm\im 2\psi_0}}-b_{\pm}\pare{1-e^{\pm\im 4\psi_0}}}/2.
		\end{split}
		\end{equation*}}
 Notice that $e^{\pm\im 2\psi_0}\neq 1$ for $\psi_0\in(0,\pi)$, and that $e^{\pm\im 4\psi_0}= 1$ if $\psi_0=\pi/2$. Then $\tilde{C}_0$ can never be zero when $\psi_0=\pi/2$.
		Suppose that $\psi_0\neq\pi/2$ and $\widetilde{C}_0=0$ for both $\pm$ signs. Then 
		\begin{equation*}
		b_{\pm}	= \frac{k^2v_0}{1+e^{\pm\im 2\psi_0}},
		\end{equation*} 
		in which case
		\begin{equation*}
		\tilde{V}(x)=\frac{k^2v_0}{2}\pare{\begin{matrix}
			x_2\tan\psi_0\\x_1\tan\psi_0-2x_2
			\end{matrix}}
		=\frac{k^2 v_0}{2}~\nabla \pare{x_1x_2\tan\psi_0-x_2^2}.
		\end{equation*}
Further, combining \eqref{eq:decayInt4} with \eqref{eq:decayInt2} 
	we arrive at \eqref{eq:decayInt3} with the constant
	\begin{equation*}
	\widetilde{C}_1=\gamma_0 \widetilde{C}_0 -k^2 v_0\rho_0 C_1,
	\end{equation*}
namely,
		\begin{equation*}
		\mp\, 2\im  e^{\pm \im 2\varphi}\widetilde{C}_1
		= 
		k^2v_0\pare{\gamma_0-\rho_0}\pare{1-e^{\pm\im 2\psi_0}}-\gamma_0 \,b_{\pm}\pare{1-e^{\pm\im 4\psi_0}}.
		\end{equation*}	
	Similar as before, we observe that $\widetilde{C}_1=0$ implies $\gamma_0=\rho_0$ if $\psi_0=\pi/2$. 
	Otherwise if $\psi_0\neq \pi/2$ then $\widetilde{C}_1=0$ for both $\pm$ signs yields
	\begin{equation*}
	\gamma_0\neq\rho_0\qquad\text{and}\qquad b_{\pm}	= \frac{k^2v_0}{1+e^{\pm\im 2\psi_0}}\pare{1-\frac{\rho_0}{\gamma_0}},
	\end{equation*}
	and as a consequence,
	\begin{equation*}
	\begin{split}
	\tilde{V}(x)
	&=\frac{k^2v_0}{2}\pare{\begin{matrix}
		-\dfrac{\rho_0}{\gamma_0}x_1+\pare{1-\dfrac{\rho_0}{\gamma_0}}x_2\tan\psi_0
		\\\pare{1-\dfrac{\rho_0}{\gamma_0}}x_1\tan\psi_0-\pare{2-\dfrac{\rho_0}{\gamma_0}}x_2
		\end{matrix}}
	\\&=\frac{k^2 v_0}{2}~\nabla \pare{\pare{1-\dfrac{\rho_0}{\gamma_0}}x_1x_2\tan\psi_0-\frac{1}{2}\frac{\rho_0}{\gamma_0}x_1^2-\frac{1}{2}\pare{2-\dfrac{\rho_0}{\gamma_0}}x_2^2}.
	\end{split}
	\end{equation*}
\end{proof}

\subsection{Do corners in 2D always scatter?}\label{Sec:CornerScat}
Now we return our attention to the scattering problem governed by \eqref{eq:MainGov1}. We first introduce the mathematical definition of the corners in the  support of the inhomogeneity we are concerned with.  Roughly speaking, we are able to deal with  constitutive material properties $a$ and $c$, whose support of the contrast  to the background, i.e. $\supp (c-1)$ or  $\supp (\condt-1)$, contains a convex corner which could be small.  As a particular case when $\condt\equiv 1$, our results recover those proven in \cite{BPS14}, \cite{ElH15} and \cite{ElH18}.  We require some regularity of $a$ and $c$ locally around the corner. We do not need to impose any additional  assumptions on $a$ and $c$ elsewhere (see Figure \ref{conf}).

\begin{defn}\label{def:corner}
	A function $f$ is said to have a corner at its support if the following is satisfied: Let a simple connected domain $\Omega\in \RR^2$ be such that $\supp f \subseteq \overline{\Omega}$. There is a point $x_0\in \partial\Omega$, a ball $B_{\epsilon}(x_0)$ of radius $\epsilon>0$ centered at $x_0$ and a cone $\mathcal{C}(x_0):=\{x\in {\mathbb R}^2:  \widehat{x-x_0} \in {\mathcal K}\}$ with ${\mathcal K}=\left\{(\cos\psi,\sin\psi);\ 0<\psi<\psi_0\right\}$, such that 
	\begin{equation*}
	\overline{\Omega}\cap B_{\epsilon}(x_0)=\supp f \cap  B_{\epsilon}(x_0) =\mathcal{C}(x_0)\cap  B_{\epsilon}(x_0):=\mathcal{C}_\epsilon(x_0).
	\end{equation*}
	In this case, we call $x_0=[x_0; \mathcal{C}_{\epsilon}(x_0)]$ a {\it corner  of (the support of) $f$ of radius $\epsilon$}.
\end{defn}

\begin{defn}\label{def:corReg}
	Given constitutive material properties  $a\in L^{\infty}(\RR^2)$, $c\in L^{\infty}(\RR^2)$, suppose that $x_0=[x_0; \mathcal{C}_\epsilon(x_0)]$ is a corner of either $a-1$ or $c-1$ with radius $\epsilon$.
	The corner $x_0$ is called {\it regular}, with respect to $a$ and $c$, if there {exist $\gamma\in L^{\infty}(\RR^2)$ and  $\rho\in L^{\infty}(\RR^2)$ satisfying the following:
	\begin{enumerate}
		\item There is a constant $\varepsilon_0>0$ such that $\gamma\in H^{3,1+\varepsilon_0}$ and $\rho\in H^{1,1+\varepsilon_0}$.
		\item $\gamma|_{\mathcal{C}_{\epsilon}}=\condt|_{\mathcal{C}_{\epsilon}}\ $ in $\ H^{3,1+\varepsilon_0}(\mathcal{C}_{\epsilon})\cap L^{\infty}(\mathcal{C}_{\epsilon})\ \ $ and $\ \ \rho|_{\mathcal{C}_{\epsilon}}=c|_{\mathcal{C}_{\epsilon}}\ $ in $\ H^{1,1+\varepsilon_0}(\mathcal{C}_{\epsilon})\cap L^{\infty}(\mathcal{C}_{\epsilon})$.
		\item  There are constants $\gamma_0,\rho_0$ and  some $\sigma>0$ such that
		\begin{equation}\label{eq:gamma}
		\pare{\gamma(x)-1}\gamma^{-1/2}(x)= \gamma_0 +O(|x-x_0|^{\sigma}),
		\end{equation}
		and
		\begin{equation}\label{eq:rho0}
		(\rho(x)-1)\gamma^{-1/2}(x)= \rho_0+O(|x-x_0|^{\sigma}),
		\end{equation}
		for almost all $x\in \mathcal{C}_{\epsilon}(x_0)$. 
		
	\end{enumerate}}
\noindent Moreover, by an abuse of terminology, we say that there is a conductivity jump (for $a(x)$) at $x_0$ if $\gamma_0\neq 0$, or a potential jump (for $c(x)$) if $\rho_0\neq 0$.
\end{defn}
{
\begin{rem}
	We note here that the first listed condition in Definition~\ref{def:corReg} suffices for $\gamma$ and $\rho$ to satisfy the assumptions in Proposition~\ref{prop:CGO} and Condition~\ref{cond1} with $n=2$ and $s=0,1$. 
	In particular, we can take $\tilde{p}=1+\varepsilon_0/2$ and $p\ge3(1+2\varepsilon_0)/(1-4\varepsilon_0)>3$.
	The property \eqref{eq:condqf} now holds by taking $1/p=1/\tilde{p}-1/(1+\varepsilon_0)$ 
	and using the H\"{o}lder's inequality $\|gh\|_{L^{\tilde{p}}}\le C\|g\|_{L^p}\|h\|_{L^{1+\varepsilon_0}}$.
\end{rem}
}
\medskip 
\noindent
Theorem \ref{thm:cornerscatt} and Theorem \ref{thm:cornerscattPot} in the following state our main results concerning the lack of non-scattering phenomena.
{To this end, we give a class of incident fields for which we cannot conclude yet whether or not they will be scattered by inhomogeneities with corners.
\begin{defn}\label{defn:Admissible}
		Given a corner of aperture $\psi_0$ and an incident field $u^{\In}$, denote $N\in\mathbb{N}$ as the order of the first nonzero term from the Taylor expansion of $\nabla u^{\In}$ at the corner.
		We say that the pair $(u^{\In},\psi_0)$ belongs to the class $\mathscr{E}$ if there holds $N=(l/\psi_0)\,\pi-1$ with some positive integer $l$.
\end{defn} }
\begin{thm}[Conductivity corner scattering]\label{thm:cornerscatt}
	Given $\condt, c\in L^{\infty}(\RR^2)$ satisfying \eqref{eq:gammEllip} and \eqref{eq:gammBound}, let $u=u^{\In}+u^{\Sc}$ be the total field of the scattering problem \eqref{eq:MainGov1}--\eqref{eq:Radiat}.	
	Suppose that there is a corner $x_0=[x_0; \mathcal{C}_{\epsilon}(x_0)]$ at the support of $a-1$ which is regular with respect to $a$ and $c$ in the sense of Definition~\ref{def:corReg}. 
	Assume further that there is a conductivity jump for $a$ at $x_0$, and that at $x_0$,  either $u^{\In}$ vanishes or both $u^{\In}$ and $\nabla u^{\In}$ are nonzero.		
	Then the scattered field $u^{\Sc}$ cannot be identically zero in the exterior of any bounded ball in $\RR^2$ except, \textbf{perhaps}, if {$(u^{\In},\psi_0)$ belongs to the class $\mathscr{E}$ with $\psi_0$ the aperture of $\mathcal{C}_{\epsilon}(x_0)$}.
%
%
\end{thm}
\begin{proof}
	We prove this result by contradiction. Assume, up to a rigid change of coordinates, that $x_0$ locates at the origin.
	Suppose that $u^{\Sc}$ is identically zero outside some bounded ball. Then by unique continuation, $u^{\Sc}$ is zero in $\RR^{2}\setminus\overline{\Omega}$ for any Lipschitz domain $\Omega$  as in Definition~\ref{def:corner} for $a-1$.
	As a consequence, the interior transmission eigenvalue problem \eqref{eq:ITPpde} and \eqref{eq:ITPbc} are satisfied for $u$ the total field and $v=u^{\In}$ the incident field of the scattering problem.
	In particular, the local problem \eqref{eq:TEVPlocal} is satisfied with functions $\gamma$ and $\rho$ as in Definition~\ref{def:corReg} for $\condt$ and $c$, respectively. 
	
	Given $d\in\mathcal{K}'_{\delta}$ with a fixed constant $\delta>0$, for any positive $\tau$ sufficiently large, we can find from  Proposition~\ref{prop:CGO} a solution $w$ to \eqref{eq:PDE} that is of the form \eqref{eq:CGO} with the residual $r$ satisfying 
	\begin{equation}
	\|r\|_{H^{s,p_{s}}}=o(\tau ^{s-2/p_s}),\quad s=0,1,
	\end{equation}
	where $p_0$ and $p_1$ are those constants as specified in Definition~\ref{def:corReg}.
	Let the vector field $\tilde{V}=\tilde{V}(x)=V(\hat{x})|x|^N$ be the first nonzero (the $N$-th) term 
	from the Taylor expansion of $\nabla u^{\In}$ at the corner. 
	Then we can always write $v=u^{\In}$ in the form \eqref{eq:Expanv0} around the corner. Moreover, our assumptions on vanishing/nonvanishing of  $u^{\In}$ and $\nabla u^{\In}$ at the corner imply $N_0\ge N$, according to Lemma~\ref{lem:TaylorvV}. Letting $\gamma_0$ be the constant defined in Definition~\ref{def:corReg} for $\gamma$ or $a$, we denote the integral
	\begin{equation}\label{eq:IntI}
	I:=\condcor_0\int_{\mathcal{C}_{\epsilon}} \tilde{V}(x)\cdot \eta e^{\eta\cdot x} \,dx 
	.
	\end{equation}	
	Then we know from \eqref{eq:IntVanish3} in Corollary~\ref{Cor:ContraSim} that 
	$I=o(\tau^{1-n-N})$, where $n=2$ is the space dimension. However, Lemma~\ref{lem:DecayExact} implies that $I=\gamma_0 C_0\tau^{1-n-N}+o(\tau e^{-\epsilon \tau/2})$. These two asymptotics can not both be true unless $C_0=0$, namely, when \eqref{eq:Nexept} holds. 
\end{proof}

\begin{rem}\label{rem:Thm5.1-2}
	In Theorem \ref{thm:cornerscatt} we are confined with the case when $u^{\In}(x_0)=0$ or when $u^{\In}(x_0)\neq0$ and $\nabla u^{\In}(x_0)\neq0$. In fact, the complementary situation, i.e., when $u^{\In}(x_0)\neq 0$ and  $\nabla u^{\In}(x_0)=0$, can be dealt with by similar arguments. 
	In this case, we have $N=1$ and $N_0=1$. As a counterpart of \eqref{eq:IntI}, we will have 
	\begin{equation*}
	\abs{\condcor_0\int_{\mathcal{C}_{\epsilon}} \tilde{V}(x)\cdot \eta e^{\eta\cdot x} \,dx -k^2\rho_0 v_0 \int_{\mathcal{C}_{\epsilon}} e^{\eta\cdot x}dx}
		=o(\tau^{-n}).
	\end{equation*}
	However, Lemma~\ref{lem:DecayExact3} implies that this cannot be true and hence $u^{\In}$ is always scattered, except (perhaps) for some specific cases. In particular, if $\psi_0=\pi/2$, then the only possible nonscattering case is when $c(x_0)=a(x_0)$; and if $\psi_0\neq\pi/2$, the only possible nonscattering case is when $c(x_0)\neq a(x_0)$ and $u^{\In}$ takes a specific form depending  on $\psi_0$ and $\rho_0/\gamma_{0}$ which can be derived from \eqref{eq:TaylorVfirst}. 
\end{rem}

\noindent
	Next, we give some consequent results of Theorem~\ref{thm:cornerscatt}, which shows that for the case when $\psi_0=\pi/2$ a wide class of incident waves of interest in applications always scatter.
\begin{cor}
	Assume that the constitutive material properties $a$ and $c$  satisfy the assumptions in Theorem~\ref{thm:cornerscatt}.  Then, for the right corner  $\psi_0=\pi/2$, any incident field $u^{\In}$ that satisfies  either one of the following conditions:
	\begin{enumerate}
		\item $u^{\In}(x_0)\neq 0$ and $\nabla u^{\In}(x_0)\neq 0$,
		\item $u^{\In}(x_0)\neq 0$, $\nabla u^{\In}(x_0)= 0$ and in addition, $a(x_0)\neq c(x_0)$,
	\end{enumerate}
	must scatter.	
\end{cor}
\begin{proof}
	We adopt the notations in the proof of Theorem~\ref{thm:cornerscatt}.
	The first condition is equivalent to $N_0=N=0$. Recalling Definition~\ref{defn:Admissible} for the case of $N=0$ and $\psi_0=\pi/2$, then Theorem~\ref{thm:cornerscatt} already implies that $u^{\In}$ is always scattered. When the second condition is true, we have $N_0=0$ and $N=1$. Then from the discussion in Remark~\ref{rem:Thm5.1-2} we have that $u^{\In}$ is always scattered unless, perhaps, when $c(x_0)=a(x_0)$.
\end{proof}

\begin{thm}[Potential corner scattering]\label{thm:cornerscattPot}
	Given $\condt, c\in L^{\infty}(\RR^2)$ satisfying \eqref{eq:gammEllip} and \eqref{eq:gammBound}, let $u=u^{\In}+u^{\Sc}$ be the total field of the scattering problem \eqref{eq:MainGov1}-\eqref{eq:Radiat}.	
	Suppose that there is a corner $x_0=[x_0; \mathcal{C}_{\epsilon}(x_0)]$ at the support of $c-1$ which is regular with respect to $a$ and $c$ in the sense of Definition~\ref{def:corReg}. 
	Let $\gamma$ be the function as in Definition~\ref{def:corReg} corresponding to $\condt$. Assume further that there is a potential jump for $c$ at $x_0$.		
	Then the scattered field $u^{\Sc}$ cannot be trivially zero in the exterior of any bounded ball in $\RR^2$ if any of the following conditions is satisfied:
	\begin{enumerate}
		\item For all $x\in \mathcal{C}_{\epsilon}(x_0)$ and some constant $\sigma>0$
		\begin{equation}\label{eq:gamma2}
		\pare{\gamma(x)-1}\gamma^{-1/2}(x)=O(|x-x_0|^{2+\sigma}).
				\end{equation}
		\item For all $x\in \mathcal{C}_{\epsilon}(x_0)$ and some constant $\sigma>0$
		\begin{equation}\label{eq:gamma1}
		\pare{\gamma(x)-1}\gamma^{-1/2}(x)=O(|x-x_0|^{1+\sigma}),
		\end{equation}
		and  $N_0=N$, where $N_0$ and $N$ are the degrees of the first nonzero term from the Taylor expansion of $u^{\In}$ and $\nabla u^{\In}$, respectively, at the corner. 
		\item For all $x\in \mathcal{C}_{\epsilon}(x_0)$ and some constant $\sigma>0$
		\begin{equation}\label{eq:gamma0}
		\pare{\gamma(x)-1}\gamma^{-1/2}(x)=O(|x-x_0|^{\sigma}),
				\end{equation}
(i.e. $\gamma_0=0$, where $\gamma_0$ is defined in Definition~\ref{def:corReg}), and $u^{\In}(x_0)\neq 0$ and $\nabla u^{\In}(x_0)=0$. 
\end{enumerate}
\end{thm}
\begin{rem}
	We note here that the conditions \eqref{eq:gamma2}, \eqref{eq:gamma1} or \eqref{eq:gamma0}, essentially describe the order  of vanishing  at the corner of  $\gamma-1$, or in other words  of the contrast $a-1$ at the corner. {As a consequence, Theorem~\ref{thm:cornerscattPot}, in particular in the case of \eqref{eq:gamma2}, generalizes} the previous results proven in \cite{BPS14}, \cite{ElH15} and \cite{ElH18} for the scattering problem where $a\equiv 1$. 
\end{rem}
\begin{proof}[Proof of Theorem~\ref{thm:cornerscattPot}]
	We first follow the proof of Theorem~\ref{thm:cornerscatt} and the notations therein up to \eqref{eq:IntI}. 
	Then we let the $v_{N_0}(x)=\tilde{v}(x)|x|^{N_0}$ be the first nonzero (the $N_0$-th) term from the Taylor expansion of $u^{\In}$, and let $\rho_0\neq 0$ be the constant defined in Definition~\ref{def:corReg} for $\rho$ or $c$. Lemma~\ref{lem:DecayExactv} implies that one can alway find a direction $d\in\mathcal{K}'_{\delta}$ and a constant $C_{1,N_0}\neq 0$, which satisfy 
	\begin{equation}\label{eq:vanish0}
	\tilde{I}:=k^2\rho_0\int_{\mathcal{C}_{\epsilon}} v_{N_0}(x)e^{\eta\cdot x} dx
	=C_{1,N_0}k^2\rho_0\tau^{-n-N_0}+o\pare{\tau e^{-\epsilon \tau /2}}.
	\end{equation}
	On the other hand, applying Proposition~\ref{prop:Contra}, in particular the estimate \eqref{eq:IntVanishpho} with $\beta_2=0$, we obtain 
	\begin{equation}\label{eq:vanish}
	\begin{split}
	\tilde{I}=
	o\pare{\tau^{-n-N_0}}
	+\|\condcor_{\beta}\|_{L^{\infty}(\mathcal{K})}\,O\big(\tau^{-n-(N-1+\beta_1)}\big)
	+o\big(\tau^{-n-(N-1+\beta_1)}\big),
	\end{split}	
	\end{equation}
	where the function $\condcor_{\beta}$ and the constant $\beta_1$ are chosen for $\gamma$ to satisfy \eqref{eq:ExpanG}.
	Recall from Lemma~\ref{lem:TaylorvV} that $N_0\ge N$ or $N_0=0$ and $N=1$.

\noindent
	In the first case when \eqref{eq:gamma2} holds true, we can take $\beta_1=2$ and $\gamma_{\beta}\equiv 0$ in \eqref{eq:ExpanG}. Since $N\le N_0+1$, we obtain from \eqref{eq:vanish} that $\tilde{I}=o\pare{\tau^{-n-N_0}}$, which contradicts \eqref{eq:vanish0}.
	When the second condition is valid, namely, if $N_0=N$ and there holds \eqref{eq:gamma1}, then setting $\beta_1=1$ and $\gamma_{\beta}\equiv 0$ in \eqref{eq:ExpanG} we arrive at the same contradiction $\tilde{I}=o\pare{\tau^{-n-N_0}}$ against \eqref{eq:vanish0}.
	Lastly in the third case, we have $N_0=0$ and $N=1$. Taking $\beta_1=0$ and $\gamma_{\beta}=\gamma_0= 0$ in \eqref{eq:ExpanG} lead to the same contradiction as before. The proof is completed.
\end{proof}

{We end  this section with a remark on the ``exclusive'' corners and incident waves in Theorems~\ref{thm:cornerscatt} and \ref{thm:cornerscattPot}. As seen from the statement of these results, in the most general settings, there are particular conductivity or potential corners and related special incident fields for which we cannot conclude  that the corresponding scattered field is non-zero. At this time we don't know whether these exceptions are artifact of our technique or a  more fundamental issue arising from the presence of the contrast in conductivity near the corner, as we were not able to construct a counter example of a non-scattering corner along with the corresponding incident field.}

\section{Applications to inverse scattering for inhomogeneous media}\label{applic}
In this section we present some applications of the above corner scattering results to  inverse scattering theory for two dimensional inhomogeneous media. More precisely we consider the scattering problem (\ref{eq:MainGov1}) with $n=2$, where the constitutive material properties $a$ and $c$  defined in the beginning of Section \ref{form}.
\subsection{A global uniqueness theorem}\label{unique}
{We consider the inverse problem of determining the convex hull of (the support of) an inhomogeneity from the scattered data. We prove that the polygonal convex hull of certain inhomogeneities can be uniquely determined by a single far-field measurement.} 
Our uniqueness result 
extends the ones proven in \cite{ElH18} and \cite{HSV16}. 
We start by defining the admissible set of the inhomogeneities. 
\begin{defn}[Admissible inhomogeneities] {Given constitutive material properties  $a\in W^{1,\infty}(\RR^2)$,  $c\in L^\infty(\RR^2)$, $a(x)\geq a_0>0$, and $D$ the convex hull of $\supp(c-1)\cup \supp(a- 1)$, the inhomogeneity $(a,c,D)$ is called  {\em admissible} if it satisfies the following properties:
\begin{enumerate}
\item The convex hull $D$ is a polygon. 
\item Each corner $x_0$ of the polygon $D$ is a corner for $c-1$ as in  Definition \ref{def:corner}  (which may or may not be a corner for $a-1$); in particular, there exist a cone $\mathcal{C}(x_0)$ of aperture $\psi_0$ and a constant $\epsilon>0$ such that 
$$\mathcal{C}_\epsilon=\mathcal{C}_\epsilon(x_0):=\mathcal{C}(x_0)\cap  B_{\epsilon}(x_0)=\supp (c-1) \cap  B_{\epsilon}(x_0).$$
\item At each corner $x_0$ of $D$, there exist constant $\varepsilon_0:=\varepsilon_{0,x_0}>0$ and functions $\gamma=\gamma_{x_0}\in H^{3,1+\varepsilon_0}(\RR^n)\cap L^{\infty}(\RR^n)$ and $\rho=\rho_{x_0}\in H^{1,1+\varepsilon_0}(\RR^n)\cap L^{\infty}(\RR^n)$ satisfying $$a|_{{\mathcal C}_{\epsilon}}=\chi_{{\mathcal C}_{\epsilon}} \gamma \quad\text{and}\quad  c|_{{\mathcal C}_{\epsilon}}=\chi_{{\mathcal C}_{\epsilon}} \rho,$$  
where $\chi_{{\mathcal C}_{\epsilon}}$ denotes the characteristic function of the set ${\mathcal C}_\epsilon$. 
Moreover, there are constants $\sigma=\sigma_{x_0}>0$ and $\rho_0:=\rho_{0,x_0}\neq 0$ such that for almost all $x\in \mathcal{C}_{\epsilon}(x_0)$,
\begin{equation}\label{eq:adm_inh}
(\rho(x)-1)=\rho_0+O(|x-x_0|^{\sigma}) \quad \mbox{and}\quad(\gamma(x)-1)=O(|x-x_0|^{2+\sigma}).
\end{equation}
\end{enumerate}
We denote by ${\mathcal A}$ the set of admissible inhomogeneities $(a,c,D)$  (see Figure \ref{adm} for some examples).}
\end{defn}
\begin{figure}[h]
\centerline{\begin{tabular}{ccc}
\includegraphics[width=0.20\textwidth]{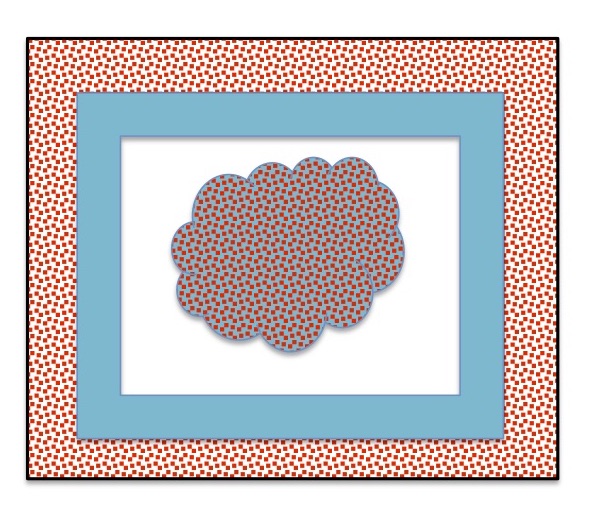} \;\; &\;\;
\includegraphics[width=0.23\textwidth]{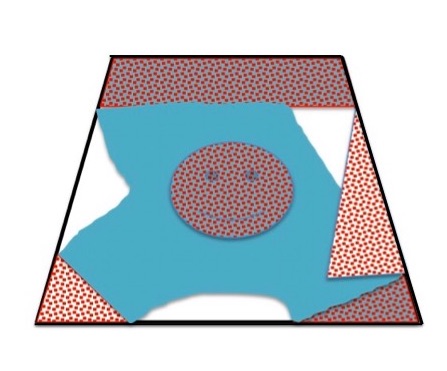}\;\; & \;\; \includegraphics[width=0.18\textwidth]{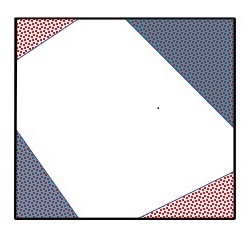}
\end{tabular}}
\caption{Examples of admissible  inhomogeneities. Dotted filling  indicates $\supp(c-1)$, uniform coloring  indicates $\supp(a-1)$ darker dotted filling indicates the  support of $\supp(c-1)\cap \supp(a- 1)$. One incident field suffices to determine the polygonal convex hull depicted by the tick line.}
\label{adm}
\end{figure}

\noindent
We can prove the following uniqueness theorem. We recall that an incident field  $u^{\In}$ is an entire solution the Helmholtz equation
$$\Delta u^{\In}+k^2u^{\In}=0\quad\mbox{in $\RR^2$}.$$
\begin{thm}\label{th:uni}
Given an admissible inhomogeneity $(a,c,D)\in {\mathcal A}$. Then the far field pattern $u^\infty$ corresponding to one single incident wave $u^{\In}$ uniquely determines the convex hull $D$ of $\supp(c-1)\cup \supp(a- 1)$.
\end{thm}
\begin{proof} Let $(a_j,c_j,D_j)\in {\mathcal A}$, $j=1,2$ be two admissible inhomogeneities, and let $u_j=u_j^{\Sc}+u^{\In}$ and $u^\infty_j\neq 0$ be the  corresponding total field and far field pattern, respectively, due to the incident field $u^{\In}$. Assume that $u_1^\infty(\hat x)=u_2^\infty(\hat x)$ for all $\hat x$ in the unit circle. Then from Rellich's Lemma the scattered fields $u_1^{\Sc}=u_2^{\Sc}$ coincide, and consequently so do the total fields $u_1=u_2$, up to the boundary of $\RR^2\setminus \overline{D_1\cup D_2}$, where we recall $D_1$ and $D_2$ are the (polygon) convex hull  of $\supp(c_1-1)\cup \supp(a_1- 1)$ and $\supp(c_2-1)\cup \supp(a_2-1)$, respectively.  If $D_1\neq D_2$, then there is a corner  $x_0:=[x_0;{\mathcal C}_{\epsilon}]$ for some small $\epsilon>0$ (to fix the idea) of $D_1$ that lies in the exterior of $D_2$  (see Figure \ref{uni}).  Hence, we have that $\nabla\cdot \condt_1 \nabla u_1 + k^2 c_1u_1=0$ in  ${\mathcal C}_{\epsilon}$,  $\Delta u_2+ k^2 u_2=0$ in ${\mathcal B}_{\epsilon}$, and $u_1=u_2$ and $\condt_1\partial_{\nu} u_1=\partial_{\nu}u_2$ at the vertices of ${\mathcal C}_{\epsilon}$.   Since, by assumption of the admissible inhomogeneities, the corner     $x_0:=[x_0;{\mathcal C}_{\epsilon}]$ satisfies the assumption $1$ of Theorem~\ref{thm:cornerscattPot} and hence by exactly the same argument as in the proof of Theorem~\ref{thm:cornerscattPot} we conclude that $u_2\equiv 0$ in $B_\epsilon$ whence  $u_2\equiv 0$ in $\RR^2$, by unique continuation  \cite{horm}. The latter means that  the (radiating) scattered field $u_2^{\Sc}=-u^{\In}$  satisfies the Helmholtz equation in $\RR^2$, therefore $u_2^{\Sc}\equiv 0$ and $u_2^\infty=0$. We arrive at a contradiction, which proves that $D_1=D_2$.
\end{proof}
\begin{figure}[h]
\centerline{
\includegraphics[width=0.25\textwidth]{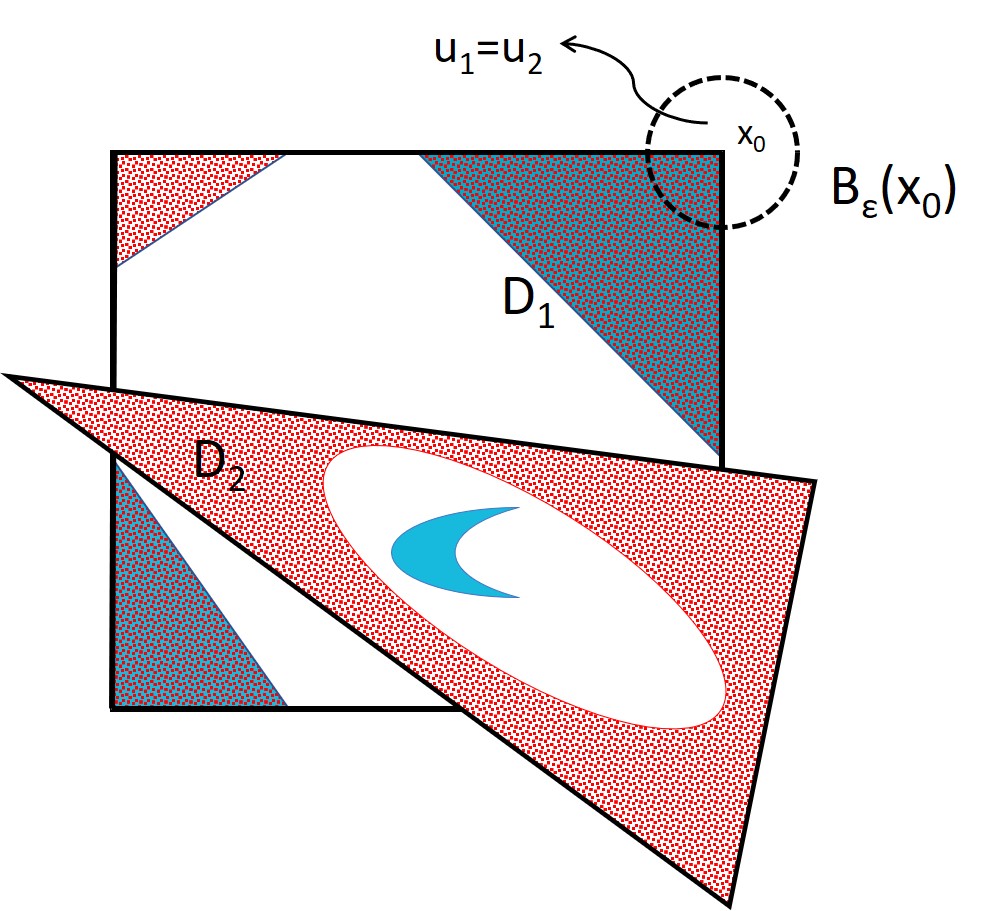}}
\caption{Intersection of two different  admissible inhomogeneities. Dotted filling  indicates the support of the contrast $c\supp(c-1)$, darker dotted filling indicates the  support of $\supp(c-1)\cap \supp(a- 1)$.}
\label{uni}
\end{figure}
Although, for simplicity of the statement, we give the uniqueness result only for inhomogeneities from the admissible class $\mathcal{A}$ as specified in Definition~\ref{def:corReg},  similar shape determination results can be shown for more general class of inhomogeneities. 
	What type of additional inhomogeneities can be included, is easily seen from the proof of Theorem~\ref{th:uni}.
	In particular, the admissible inhomogeneities in Definition~\ref{def:corReg} require the conductivity contrast to vanish to second order at the corners of the convex hull. However, we can enlarge the  admissible class   $\mathcal{A}$ by including also inhomogeneities  whose polygonal convex hull  has corners with conductivity  jump, i.e.  $\gamma$ satisfies \eqref{eq:gamma} with $\gamma_0\neq 0$  assuming in addition  that such corner has  aperture which is an irrational factor of $\pi$. Formally speaking, we can even consider any bounded inhomogeneities. In this case, if the far-field data corresponding to one single incident wave is the same for two  inhomogeneities, then we can conclude that the difference between the two corresponding convex hulls (not necessarily polygons) cannot contain any ``admissible pair of corner and total field'' as specified in Section~\ref{Sec:CornerScat}. Here, by admissible pairs we mean corners and related waves which will always be non-trivially scattered by the corner ($[x_0;{\mathcal C}_{\epsilon}]$ and $u_2$ in the proof of Theorem~\ref{th:uni} for example).

\noindent
As a particular case of  Theorem \ref{th:uni} we have the following uniqueness theorem for the support of a polygonal inhomogeneity.
\begin{cor}
Given an admissible inhomogeneity $(a,c,D)\in {\mathcal A}$, and assume further that $\supp(c-1)\cup \supp(a-1)$  is a convex polygon (i.e.  $\supp(c-1)\cup \supp(a-1)=D$). Then the far field pattern corresponding to a single incident wave uniquely determines the support of the inhomogeneity $\supp(c-1)\cup \supp(a-1)$.
\end{cor}

\subsection{Approximation by Herglotz functions}
Most of the reconstruction techniques using the linear sampling methods and transmission eigenvalues depends on denseness properties of  the so-called Herglotz functions, which are entire solutions to the Helmholtz equation defined by 
$$v_g(x):=\int_{{\mathbb S}^{n-1}}g(d) e^{ikx\cdot d}ds_d,\qquad g\in L^2({\mathbb S}^{n-1}),$$
where ${\mathbb S}^{n-1}:=\left\{x\in \RR^n:\, |x|=1\right\}$, and $g$ is referred to as kernel of the Herglotz function $v_g$. 
It is well-known (see e.g. \cite{CakoniColtonHaddar2016}) that the set
$$\{v_g: \, g\in L^2({\mathbb S}^{n-1})\}$$
is dense in 
$$\{v\in {H}^1(\Omega): \, \Delta v+k^2v=0\quad \mbox{in} \; \Omega\}$$
with respect to the ${H}^1(\Omega)$-norm, where  $\Omega\in {\mathbb R}^n$ is a bounded region with connected complement. 

\noindent
Given the inhomogeneity  $(a,c,\Omega)$ defined at the beginning of Section \ref{form}, let $k>0$ be a transmission eigenvalue, i.e. the following problem 
\begin{eqnarray*}
\nabla\cdot \condt \nabla u + k^2 c u=0,&  \Delta v + k^2 v=0,&\quad\mbox{in $\Omega$},\\
u=v,&  \condt\partial_{\nu} u=\partial_{\nu}v,&\quad\mbox{on $\partial\Omega$},
\end{eqnarray*}
has nonzero solution $u,v\in H^1(\Omega)$. Our corner scattering analysis in the two dimensional case yields the following result, which concerns the approximation of the eigenfunction $v$ by Herglotz functions.  To this end,  at a transmission eigenvalue $k>0$, let the sequence of Herglotz functions $\left\{v_{g_\epsilon}\right\}$ approximate the eigenfunction $v$, i.e. 
\begin{equation}\label{eq:approx}
\lim\limits_{\epsilon \to 0}\|v_{g_\epsilon}-v\|_{H^1(\Omega)}=0.
\end{equation}
\begin{lem}\label{approx}
Assume that $\Omega\subset {\mathbb R}^2$ has a corner $x_0=[x_0; \mathcal{C}_\epsilon(x_0)]$  {for  $a-1$ with the assumptions of Theorem \ref{thm:cornerscatt} for corner aperture $\psi_0\notin \{p\pi;\ p\in\mathbb{Q}\cap (0,1)\}$, or for $c-1$ with the assumptions of  Theorem \ref{thm:cornerscattPot} with condition $1$. Let $v$ be an eigenfunction, and let the sequence of Herglotz functions $\left\{v_{g_\epsilon}\right\}$ satisfy \eqref{eq:approx}. Then $\lim\sup\|g_\epsilon\|_{L^2({\mathbb S}^1)}=\infty$.}
\end{lem}
\begin{proof}
Assume to the contrary that $\left\{\|g_\epsilon\|_{L^2({\mathbb S}^1)}\right\}$ is bounded. Then up to a subsequence $g_\epsilon\rightharpoonup g\in L^2({\mathbb S}^1)$ weakly as $\epsilon \to 0$.   Obviously $v_{g_\epsilon}\to v_g$ in $C^1(\overline{\Omega})$ and thus $v:=v_g|_{\Omega}$ which means that $v_g$ does not scatter. This contradicts the assumptions, and the lemma is proven.
\end{proof}
We remark that  for the case of $a\equiv 1$ in \cite{BlL17} the authors have shown that if  the transmission eigenfunction $v$ is approximated by a sequence of Herglotz functions with certain growing conditions, then $v$ must vanish at the corner. Indeed our analysis for ``potential corner" shows that $v$ has to vanish at any order at the corner and by analyticity be identically zero.  The exceptional cases due to the presence of the contrast $a-1$ stated in Theorem \ref{thm:cornerscatt} and Theorem \ref{thm:cornerscattPot} describe necessary vanishing properties at the corner  of the transmission eigenfunction $v$ if it can be approximated  by a sequence of Herglotz functions with uniformly bounded kernel, which is equivalent to $k$ being a non-scattering wavenumber. However, these are not sufficient conditions 
for the latter to occur.

\section{Conclusions}
We conclude the paper with a few remarks. Firstly, our construction of CGO solutions and their use to study local behavior of solutions of concerning PDEs near the vertex of a generalized corner in any dimension higher than one lays out the needed analytical framework to study corner scattering. Although, here for sake of presentation, the latter is carried out only in two dimensional case,  we strongly believe that the analogue is true for conical corners in dimension three. Moreover, similar techniques are expected to be developed to analyze edge scattering in three dimensions. If proven, such results can then be used to obtain similar uniqueness theorem as in Section \ref{unique} for polyhedral convex hull of the support of inhomogeneity in ${\mathbb R}^3$.  

Secondly, we are  perplexed by the exceptional corners in the case of contrast in conductivity. We don't know yet whether this is a shortcoming of our approach or is a more essential continuation question related to this case.  Unfortunately, for geometries with corners even in ${\mathbb R}^2$ it is hard to get simple explicit calculations for the transmission eigenvalue problem in order to see if for any of  such exceptional corners the eigenfunction corresponding to the equation of the background can be extended outside the corner, i.e. to conclude that  corner does scatter. In order to have a different angle of investigation to this issue, in a forthcoming study, we consider singularity analysis on the pair of the solution to the interior transmission problem near a generalized corner,  following the lines of  \cite{ElH18}. We are hoping to perform this singularity analysis for anisotropic conductivity coefficient also, for which the construction of CGO solutions is more complicated.

\section*{Acknowledgments}
The research of F. Cakoni is partially supported  by the AFOSR Grant  FA9550-20-1-0024 and  NSF Grant DMS-1813492.

\end{document}